\documentclass[pdftex,11pt,a4paper,oneside]{article}
\DeclareUnicodeCharacter{221A}{$\sqrt{}$}
\DeclareUnicodeCharacter{22C9}{$\rtimes$}
\usepackage[margin=2.5cm]{geometry}
\usepackage[british]{babel}
\usepackage[utf8]{inputenc}
\usepackage{amsfonts}
\usepackage{mathtools}
\usepackage{amsthm}
\usepackage{amssymb}
\usepackage{hyperref}
\usepackage{blkarray}
\usepackage{tikz-cd}
\usepackage{caption}
\usepackage{subcaption}
\usepackage{authblk}
\usepackage{appendix}
\usepackage{csquotes}
\theoremstyle{plain}
\newtheorem{theorem}{Theorem}[section]

\newtheorem{lemma}{Lemma}[section]
\newtheorem{proposition}{Proposition}[section]

\newtheorem{remark}{Remark}[section]
\usepackage{enumitem}
\newtheorem{definition}{Definition}[section]
\renewcommand{\epsilon}{\varepsilon}

\newcommand{\SLC}{\text{SL}_2\left(\mathbb{C}\right)}
\newcommand{\PSLC}{\text{PSL}_2\left(\mathbb{C}\right)}
\newcommand{\SL}{\text{SL}_2\left(\mathbb{R}\right)}

\newcommand{\PSL}{\text{PSL}_2\left(\mathbb{R}\right)}

\newcommand{\PSLo}{\text{PSL}_2\left(\mathfrak{o}_d\right)}
\newcommand{\ch}[1]{\left\lfloor #1 \right\rceil}

\newcommand{\UT}{\mathrm{T}^1}

\newcommand{\abs}[1]{\left\lvert#1\right\rvert}
\newcommand\blfootnote[1]{%
  \begingroup
  \renewcommand\thefootnote{}\footnote{#1}%
  \addtocounter{footnote}{-1}%
  \endgroup
}

\title{Complex continued fractions, Kleinian and extremal theory for cusp excursions}
\author[1]{Alexander Baumgartner\thanks{Corresponding author.}}
\author[2]{Mark Pollicott}
\affil[1]{{\small Current Address: Centro di Ricerca Matematica \enquote{Ennio De Giorgi}, Scuola Normale Superiore di Pisa, Italy} }
\affil[1,2]{University of Warwick, Mathematics Institute, United Kingdom}
\affil[ ]{{alexander.baumgartner@sns.it}\\{masdbl@warwick.ac.uk}}
\usepackage[
backend=biber,
style=math-numeric,
sorting=anyt,
url=false,
doi=false,
isbn= false,
dashed=false
]{biblatex}
\begin{document}
\maketitle

\blfootnote{
keywords: Bianchi group, continued fraction, Diophantine approximation, cusp excursion, ergodic theory\\ This research is part of a project that has received funding from the European Union’s Horizon 2020 research and innovation programme under ERC Advanced Grant No. 833802-Resonances.}
\abstract{For the each of the five Euclidean rings of complex quadratic integers, we consider a complex continued fraction algorithm with digits in the ring. We show for each algorithm that the maximal digit obeys a Fréchet distribution. We use this to find a limiting distribution for cusp excursions on Bianchi orbifolds associated with the aforementioned rings of quadratic integers.}

\section{Introduction}

\raggedright The goal of the first part of this paper is to generalise an extreme value theorem for continued fraction expansions of real numbers to the complex case. In general a \textbf{continued fraction} is an expression of the form 
\begin{equation}\label{equation: definition of continued fraction}
    z=a_0+\frac{1}{ a_1+ \frac{1}{ a_2   +\frac{1}{ {a_3}_{\ddots}}}}.
\end{equation}
We denote the right hand side by $[a_0,a_1,a_2,\ldots]$. In the context of this paper, the digits $a_0,a_1,\ldots$ are in a ring of complex quadratic integers. In the more  familiar case of regular continued fractions where $a_n \in \mathbb N$ and $z\in (0,\infty)$, the choice of the digits $a_n$ in \eqref{equation: definition of continued fraction} is unique when $z$ is irrational. For this choice of algorithm, we have the following striking result due to Galambos \cite{Galambos1972}.
\begin{proposition}[Galambos]\label{proposition: Galambos}
 For regular continued fractions, if $u>0$, then 
 \[\lim_{N\to\infty} \nu\left\{z\in[0,1]: \max_{1\leq n\leq N}a_n \leq \frac{uN}{\log2} \right\} = e^{-1/u} ,\]
 where $d\nu = (\log2)^{-1}dz/(1+z)$ is the Gauss measure on $[0,1]$.  
\end{proposition}
We obtain an analogous result for the case $z\in\mathbb{C}$ as follows. If $d=1,2,3,7$ or $11$, we can consider expressions of the form \eqref{equation: definition of continued fraction} with $a_n \in \mathfrak{o}_d$: the ring of integers of the field extension $\mathbb{Q}(\sqrt{-d})$. We generate such an expansion by requiring for all $n>0$ that the absolute value of the tail
\begin{equation}\label{equation: tails}
    [0,a_n,a_{n+1},a_{n+2},\ldots]
\end{equation}
 is as small as possible. See Section \ref{section: Complex Continued Fractions} for more details. We call this the \textbf{nearest-integer $\mathfrak{o}_d$-continued fraction expansion of $z$}. For $d=1$, this is more commonly known as the Hurwitz complex continued fraction expansion, which are treated in \cite{Hurwitz}. Let $K_1$ be the square in $\mathbb{C}$ consisting of all points with real and imaginary parts in $[-1/2,1/2]$. By \cite{Ei2019}, there exists a Borel probability measure $\mu_1$ on $K_1$ equivalent to the Lebesgue measure for which the digits of the Hurwitz continued fraction expansion form a stationary sequence. With respect to this measure, we have the following result due to Kirsebom \cite{Kirsebom}.

\begin{proposition}[Kirsebom]\label{proposition: Kirsebom}
     There exists a constant $C_1>0$ such that for any $u>0$, we have that 
    \[\lim_{N\to\infty} \mu_1\left\{z\in K_1: \max_{1\leq n\leq N}|a_n(z)| \leq C_1u\sqrt{N} \right\} = e^{-1/u^2} ,\]
    where $(a_n(z))_n$ are the digits in the Hurwitz complex continued fraction expansion of $z$.
\end{proposition}
This distribution is a special case of the \textbf{Fréchet Distribution}. The proof of Proposition \ref{proposition: Kirsebom} relies on several properties of the Hurwitz continued fraction algorithm discovered in \cite{Ei2019}. Based on recent work by Ei, Nakada and Natsui in \cite{Ei23+}, we extend Kirsebom's result to other values of $d$.

\begin{theorem}[Fréchet Law for the Maximal Digit]\label{theorem: Gumbel Law for Bianchi Groups}
    Let $d=1,2,3,7,11$. Let $K_d$ be the set of all points in $\mathbb{C}$ for which the closest element of $\mathfrak{o}_d$ is $0$. There exists a constant $C_d>0$ depending on $d$ such that for any $u>0$, we have that 
    \[\lim_{N\to\infty} \eta\left\{z\in K_d: \max_{1\leq n\leq N}|a_n(z)| \leq C_du\sqrt{N} \right\} = e^{-1/u^2} ,\]
    where $(a_n(z))_n$ are the digits in the nearest-integer $\mathfrak{o}_d$-continued fraction expansion of $z$ and $\eta$ is any probability measure that is absolutely continuous with respect to the Lebesgue measure. In particular, $C_d$ does not depend on $\eta$. 
\end{theorem}
\begin{remark}
    For simplicity of exposition, we have stated Theorem \ref{theorem: Gumbel Law for Bianchi Groups} for continued fraction digits with largest absolute absolute value. The proof also works for the digit with the $k$-th largest absolute value. In that case, we must replace the Fréchet law $u\mapsto e^{-1/u^2}$ with the Poisson law $u\mapsto e^{-1/u^2}\sum_{m=0}^k u^{-2m}/m^2$.
\end{remark}

The goal of the second part of this paper is to use Theorem \ref{theorem: Gumbel Law for Bianchi Groups} to study cusp excursions of the geodesic flow on certain three dimensional analogues of the modular surface, which date back to Bianchi \cite{Bianchi1892}.  We recall that the group of orientation-preserving isometries of the three-dimensional hyperbolic space $\mathbf{H}^3=\{x+yi+rj:x,y\in\mathbb{R}, r>0\}$ can be identified with the projective special linear group $\PSLC$, see Section \ref{section: Bianchi Orbifolds}. We call the subgroup $\PSLo$ consisting of all the matrices in $\PSLC$ with entries in $\mathfrak{o}_d$ a \textbf{Bianchi group}. The quotient space $\mathbf{P}_d =\PSLo\backslash \mathbf{H}^3$ is called a \textbf{Bianchi orbifold}. We denote the unit tangent bundle of $\mathbf{P}_d$, i.e. the suborbifold of the tangent bundle consisting of vectors of unit length, by $\UT \mathbf{P}_d$.
\begin{theorem}\label{theorem: cusp excursions on Euclidean Bianchi Orbifolds}
    Let $\mathbf{P}_d$ be a Bianchi orbifold corresponding to the imaginary quadratic field extension $\mathbb{Q}[\sqrt{-d}]$ for $d=1,2,3,7,11$. For a vector $v\in\UT\mathbf{P}_d$, let $v(t)$ be the vector obtained by applying the geodesic flow for some time $t$. Let $h: \UT\mathbf{P}_d\to [0,1) $ be the function which returns the distance to the horosphere $r=1$. There exists a constant $\kappa_d>0$ such that 
    \[\lim_{T\to\infty} m\left\{ v\in \UT\mathbf{P}_d : \sup_{0 \leq t\leq T} h(v(t)) -\frac{1}{2}\log T \leq  \log(u)+\kappa_d\right\}=e^{-1/u^2},\]
    where $m$ is the Liouville measure, i.e. the natural volume on $\UT \mathbf{P}_d$.
\end{theorem}
Furthermore, we manage to obtain an exact expression for the constant $\kappa_d$ in Subsection \ref{subsection: calculating alpha}.

\begin{remark}
    If we instead look at the $k$-th largest cusp excursion, we must again replace the Fréchet law $u\mapsto e^{-1/u^2}$ with the Poisson law $u\mapsto e^{-1/u^2}\sum_{m=0}^k u^{-2m}/m^2$.
\end{remark}
Theorem \ref{theorem: cusp excursions on Euclidean Bianchi Orbifolds} is a higher-dimensional analogue of the following result due to the second author \cite{Pollicott2007LIMITINGDF}. Let $\mathbf{M}$ be the modular surface and let $m_\mathbf{M}$ be the Liouville measure for its unit tangent bundle $\UT\mathbf{M}$.  
\begin{proposition}[Pollicott]\label{proposition: Geodesic Excursion on M}
Let $u>0$. Then
\[\lim_{T\to\infty} m_\mathbf{M}\left\{ v\in \UT\mathbf{M} : \sup_{0 \leq t\leq T} h(v(t)) -\log T \leq \log\left(\frac{3u}{\pi^2}\right)\right\}=e^{-1/u},\]
where $h$ returns to distance to the horocycle $y=1$.
\end{proposition}
These results are refinements of Sullivan's logarithm law \cite{Sullivan} for the aforementioned Bianchi orbifolds and the modular surface. 
\begin{proposition}[Sullivan]
     Let $\Gamma$ be a discrete subgroup of isometries of the $n$-dimensional hyperbolic space $\mathbf{H}^n$. Assume $\Gamma\backslash\mathbf{H}^n$ is of finite volume but not compact and denote for $v,w\in \Gamma\backslash\mathbf{H}^n$ the hyperbolic distance between the respective base points in $\Gamma\backslash\mathbf{H}^n$ by $d(v,w)$. Then
\[\limsup_{t\to\infty} \frac{d(v,v(t))}{\log t} = \frac{1}{n-1},\]
for almost all vectors $v$ with respect to the Liouville measure in the unit tangent bundle of $\Gamma\backslash\mathbf{H}^n$. 
\end{proposition}

When $d=1$, the fractions generated by the nearest-integer $\mathfrak{o}_d$-continued fraction expansion of $z$ are best approximations for Lebesgue-almost every $z$. This is no longer the case for other values of $d$ \cite{Lakein1973}. However, establishing the connection with cusp excursions allows us to nonetheless obtain a Diophantine approximation result.
\begin{theorem}\label{theorem: Diophantine Approximation}
    Consider the imaginary quadratic field extension $\mathbb{Q}[\sqrt{-d}]$ for $d=1,2,3,7,11$. Then if $u>0$,
    \[\lim_{Q\to\infty} \nu \left\{z\in\mathbb{C}:\abs{z-\frac{p}{q}}\leq \frac{1}{2e^{\kappa_d} \abs{q}^{2}\sqrt{\log Q}u}\text{ for some } p,q\in\mathfrak{o}_d \text{ with } \abs{q}\leq Q \right\}=1-e^{-1/u^2},\]
    where $\nu$ is any measure equivalent to the Lebesgue measure on $\mathbb{C}$.
\end{theorem}
\section{Complex continued fractions}\label{section: Complex Continued Fractions}
Given \textit{any} continued fraction expansion $[a_0,a_1,\cdots]$, we define the \textbf{associated quotient pair of sequences} $(p_n)_n$ and $(q_n)_n$ by the recursive equations
\begin{equation}\label{equation: convergents}
    \begin{split}
        p_{-2} &= 0, \quad p_{-1}=1, \quad p_{n}=a_np_{n-1}+p_{n-2}\text{ and}\\
        q_{-2} &= 1, \quad q_{-1}=0, \quad q_{n}=a_nq_{n-1}+q_{n-2} \text{ for }n\geq0.
    \end{split}
\end{equation}
It is a standard result that for all $n$, 
\begin{equation}\label{equation: determinant of matrix of convergents}
    p_{n-1}q_n-p_nq_{n-1}=(-1)^n
\end{equation}
and
\[\frac{p_n}{q_n} = [a_0,a_1,\ldots,a_n].\]
We call $p_n/q_n$ the \textbf{$n$-th convergent} of the continued fraction.

The rest of this section is devoted to describing the nearest-integer $\mathfrak{o}_d$-continued fraction expansion algorithm in detail. Recall that the ring of integers $\mathfrak{o}_d$ is defined as the set of all $\mathbb{Q}[\sqrt{-d}]$-valued solutions to the quadratic equations of the form $z^2+bz+c,$ where $b,c\in\mathbb{Z}$. If $d$ is a square-free integer, then
\[\mathfrak{o}_d = \mathbb{Z}[\omega_d] = \{n+m\omega_d:n,m\in\mathbb{Z}\},\]
where $\omega_d = \frac{(-1+\sqrt{-d})}{2}$ if $d= 3 \mod 4$ and $\omega_d = \sqrt{-d}$ otherwise.

We defined the set $K_d\subset \mathbb{C}$ in Theorem \ref{theorem: Gumbel Law for Bianchi Groups} as
\[K_d = \{z\in\mathbb{C}: |z|< |z-\lambda|\text{ for all }\lambda\in\mathfrak{o}_d\backslash\{0\}\}.\]
In other words, $K_d$ is the Dirichlet fundamental domain centred at the origin for translations by elements in $\mathfrak{o}_d$. Explicitly, this is given by 
\begin{equation}
    \begin{split}
        K_d:=&\left\{x+iy:x,y\in\mathbb{R}: |x|<\frac{1}{2}, |y|< \frac{\sqrt{d}}{2} \right\} \text{if } d=1,2\text{, and}\\
        K_d:= &\Big\{x+iy:x,y\in\mathbb{R}:|x|<\frac{1}{2}, \\
        &\abs{x-y\sqrt{d}}<\frac{d+1}{4},\abs{x+y\sqrt{d}}<\frac{d+1}{4}  \Big\}\text{ if } d=3,7,11.
    \end{split}
\end{equation}
See \cite{Ei22} for more details and for complex continued fraction algorithms associated to other choices of fundamental domains.

\begin{definition}\label{definition:Bianchi Continued Fractions}
    Let $d=1,2,3,7,11$. Define the function $\ch{\cdot}_d: \mathbb{C}\backslash \{0\}\to \mathfrak{o}_d$ such that
\[\ch{z}_d = a\]
implies that $ z\in a+\overline{K_d}$.
\end{definition}

\begin{remark}
    The function $z\mapsto \ch{z}_d$ is not uniquely determined if $z\in a+\partial K_d$ for some $a\in\mathfrak{o}_d$.  Any choice will work, however, since we will only consider properties which hold almost everywhere with respect to the Lebesgue measure.
\end{remark}
We associate to the function $\ch{\cdot}_d$ the following complex generalisation of the Gauss map
\[G_d: K_d\to K_d: z\mapsto \frac{1}{z}-\ch{\frac{1}{z}}_d.\]
Define the sequences $(z_n)_n$ and $(a_n)$ by
\begin{equation}\label{equation:Algorithm}
    \begin{split}
        &z_0 = z, a_0(z) = \ch{z_0}_d,\\
        &z_n = \frac{1}{G_d^{n-1}(z-a_0(z))}\text{, }a_n(z) = \ch{z_n}_d \text{for  }n\geq 1.
    \end{split}
\end{equation}
If the expression $[a_0,a_1,\ldots]$ converges, then $G_d^{n-1}(z-a_0(z))$ is equal to the tail \eqref{equation: tails}. Since the latter is in $K_d$, the above algorithm indeed minimizes the absolute value of the tails.
\begin{proposition}\label{proposition: growth of denominator}
    The continued fraction $[a_0,a_1,\ldots]$ with $a_i$ chosen by \eqref{equation:Algorithm}, converges for all $z\in\mathbb{C}$. It stops in a finite number of steps if and only if $z\in\mathbb{Q}[\sqrt{-d}]$. If $p_n/q_n$ is the $n$-th convergent of $z$, then
    \begin{equation}\label{equation: increasing denominators}
        \abs{\frac{q_{n-1}}{q_n}}<1
    \end{equation}
    for all $n\geq 0$. The identity
\begin{equation}\label{equation: reversing the digits}
    q_n/q_{n-1} = [a_n,a_{n-1},\ldots,a_1]
\end{equation}
also holds.
\end{proposition}
The proof of Equation \eqref{equation: increasing denominators} is given in \cite{Hurwitz} for $d=1,3$. A discussion of this fact for all algorithms under consideration can be found in \cite{Lakein1973}. 
 Equation \eqref{equation: reversing the digits} is a standard result and can be proved by induction.
 \begin{remark}
      The values $d= 1,2,3,7,11$ are those for which a Euclidean division algorithm exists for $\mathfrak{o}_d$ with respect to the partial order given by the absolute value. For other values of $d$, the nearest-integer algorithm does not converge. This is discussed in, e.g. \cite{Ei22}. 
 \end{remark}

From this point on, when we refer to the continued fraction expansion of a given complex number, we exclusively mean with respect to the nearest-integer $\mathfrak{o}_d$-continued fraction algorithm for $d=1,2,3,7,11$.

\section{Extremal value theory for the digits of continued fractions}
In this section, we briefly describe how to obtain Theorem \ref{theorem: Gumbel Law for Bianchi Groups} for nearest-integer $\mathfrak{o}_d$-continued fraction expansions. With the exception of a short application of Eagleson's lemma to obtain the theorem for all equivalent measures, the proof is analogous to that of Theorem 5 in \cite{Kirsebom}.  We first establish an asymptotic result for the probability that the absolute value of a given digit exceeds some large value of $t$ in Lemma \ref{lemma: chance of large digits}. The theorem then follows from a mixing result for $\psi$-mixing sequences of random variables, the definition of which we now briefly introduce. 

\begin{definition}\label{definition: psi-mixing}
    Let $\xi_{n\in \mathbb{N}}$ be a stationary sequence of random variables on a probability space $(\Omega,\mathcal{B},\nu)$. For $l\leq m \in \mathbb{N}\cup \{\infty\}$ let $\mathcal{F}_{l}^m$ be the coarsest $\sigma$-algebra such that every random variable in $\{\xi_n\}_{n=l}^m$ is measurable. The sequence is \textbf{$\psi$-mixing} if there exists a sequence $(\alpha_n)_n$ such that $\alpha_n\to 0$ as $n\to\infty$ and
    \[\abs{\nu(A\cap B) - \nu(A)\nu(B)}\leq \alpha_n\nu(A)\nu(B) \]
    for all $m\in\mathbb{N}$, $A\in \mathcal{F}_{1}^m$ and $B\in \mathcal{F}_{m+n}^\infty$.
\end{definition}
See \cite{Bradley} for an overview of more mixing conditions.

Both the asymptotics and mixing results are based on recent work by Ei,  Nakada and Natsui, who extended their previous results with Ito for Hurwitz complex continued fractions in \cite{Ei2019} to all complex continued fraction algorithms under our consideration in \cite{Ei23+}.

The following lemma summarises some results from \cite{Ei22} and \cite{Ei23+} that we require. We assume throughout that $d=1,2,3,7,11$.
\begin{lemma}\label{lemma: bounded equivalence with Lebesgue measure}
Let $\lambda$ be the Lebesgue measure on $\mathbb{C}$. The map $G_d$ admits an ergodic invariant measure $\mu_d$ which is absolutely continuous with respect to the Lebesgue measure. Furthermore there exists a $C'>0$ such that
\begin{equation}\label{equation: bounded equivalence}
    \frac{1}{C'}\lambda(A)\leq \mu_d(A)\leq C'\lambda(A)
\end{equation}
for every Borel set $A$. In fact, there is a partition 
\[K_d = \bigcup_{k=1}^J V_k\]
into path-connected sets whose boundaries consist of a finite union of segments of circles and lines such that
\begin{equation}\label{equation: mu as marginal}
    d\mu_d = C\sum_{k=1}^Jg_k\mathbf{1}_{V_k} d\lambda(z),
\end{equation}
where $g_k$ is a real-analytic function on an open neighbourhood of $V_k$, $\mathbf{1}_{V_k}$ is the characteristic function on $V_k$ and $C$ is a normalizing constant which ensures $\mu_d$ is a probability measure.
\end{lemma}
See Figure \ref{figure: Nakada Partition and Dual} for an illustration of this partition.

\begin{figure}
    \centering
    \includegraphics[width=0.5\textwidth]{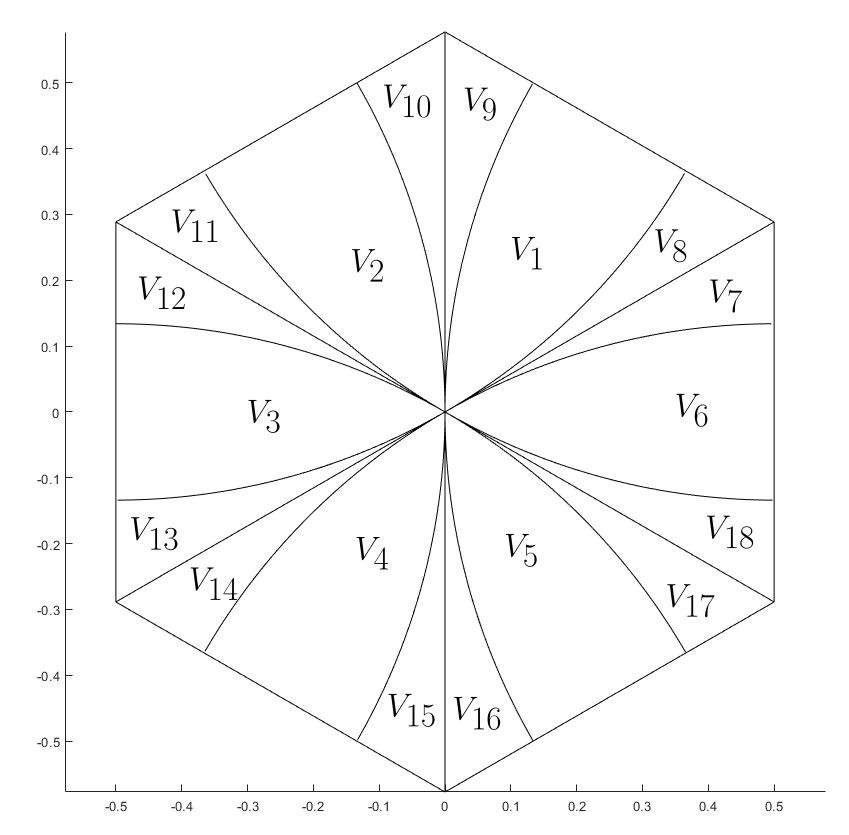}
    \caption{The partition $\bigcup_{k=1}^J V_k$ for $d=3$.}
    \label{figure: Nakada Partition and Dual}
\end{figure}
The following proof extends Lemma 2 in \cite{Kirsebom} to the nearest-integer $\mathfrak{o}_d$-continued fractions under consideration. 

\begin{lemma}\label{lemma: chance of large digits}
    Let $d= 1,2,3,7,11$. There is a constant $H>0$ such that 
    \[\mu_d\{z\in K_d: |a_1(z)|> t\} = \frac{H}{t^2} + O\left(\frac{1}{t^3}\right)\text{ as }t\to\infty, \]
    where $a_1(z)$ is the first digit in the nearest-integer $\mathfrak{o}_d$-continued fraction expansion of $z$.
\end{lemma}
\begin{proof}
Let $\iota$ be the inversion map $\widehat{\mathbb{C}}\to\widehat{\mathbb{C}}: z\mapsto 1/z$ on the Riemann sphere $\mathbb{C}\cup\{\infty\}$. Define the set $A_t$ for $t>0$ by setting
\[A_t:= \iota(\{z\in K_d: |a_1(z)|> t\}) =\{w\in \iota(K_d): |\ch{w}_d|> t\}.\]
Let $D(r)$ be the disc of radius $r$ centred around the origin and denote its complement by $D(r)^C$. It follows from the definitions of $\ch{\cdot}$ and $a_1(\cdot)$ that 
\begin{equation*}
    D(t+1)^C \subset A_t\subset D(t-1)^C.
\end{equation*}
For $t>1$. By applying $\iota$ to the formula above and taking the measure $\mu_d$ we obtain that 
\[\abs{\mu_d\{z\in K_d: |a_1(z)|> t\}-\mu_d\left(D(1/t)\right)}\leq C'\lambda\left(D\left(\frac{1}{t-1}\right)\backslash D\left(\frac{1}{t+1}\right)\right),\]
where $C'$ is the constant in \eqref{equation: bounded equivalence}. The Lebesgue measure of the annulus is of order $O(1/t^3)$, so we obtain
\[\mu_d\{z\in K_d: |a_1(z)|> t\} = \mu_d\left(D(1/t)\right) + O\left(\frac{1}{t^3}\right).\]
It thus suffices to prove that for some $H$,
\[\mu_d\left(D(1/t)\right) = \frac{H}{t^2} + O\left(\frac{1}{t^3}\right)\text{ as }t\to\infty. \]
By \eqref{equation: mu as marginal}, it is enough to show that the functions 
\[f_k(x) =  \int\mathbf{1}_{V_k\cap D(x)}(z)g_k(z)d\lambda(z) \]
satisfy $f_k\left(\frac{1}{t}\right) =\frac{W_k}{t^2} + O\left(\frac{1}{t^3}\right)$ for some constant $W_k$. Since $f_k(x) = O(1/t^2) $, it suffices by Taylor's theorem to show that $f_k$ is three times differentiable in $[0,\epsilon)$ \footnote{Here the differentiability of a function at $0$ means existence of the right derivative.} for some $\epsilon>0$. Recall that the boundary of $V_k$ consists of a finite union of circle segments and line segments. For sufficiently small $x$, the disk $D(x)$ intersects two segments if $0$ is a vertex of $V_k$. For example, if $d=3$, this is the case for all the partition elements in Figure \ref{figure: Nakada Partition and Dual}. Else $D(x)$ intersects at most one line segment. In any case, these segments admit a smooth parametrisation of the form $r\mapsto re^{i\theta(r)}$ for small $r$. Hence we can write the above integral in polar coordinates and obtain 
\[f_k(x) = \int_0^x\int_{\theta_1(r)}^{\theta_2(r)}g_k(re^{i\theta}) rd\theta dr,\]
where $\theta_1$ and $\theta_2$ vary smoothly with $r$. Since $g_k$ is real analytic on an open neighbourhood of $V_k$, we see by repeated applications of the Leibniz integral rule that $f_k$ is infinitely differentiable in $[0,\epsilon)$ for some $\epsilon>0$ small enough.
\end{proof}
We now state the mixing property in \cite{Ei23+} that we require.
\begin{proposition}[Ei, Nakada, Natsui]\label{proposition: exponential decay}
    Let $\mu_d$ be as in Lemma \ref{lemma: bounded equivalence with Lebesgue measure}. There exist constants $\alpha>0$ and $0<\rho<1$ depending on $d$ such that for all $k,m\in\mathbb{N},  b_1,b_2,\ldots,b_k\in \mathfrak{o}_d$ and $ c_1,\ldots,c_m\in \mathfrak{o}_d$ we have
    \begin{equation}\label{equation: mixing map}
    \begin{split}
        &|\mu_d(\langle b_1,\ldots,b_k \rangle \cap G_d^{-(n+k)} \langle c_1,\ldots,c_m\rangle) \\
        &-\mu_d(\langle b_1,\ldots,b_k \rangle)\mu_d( \langle c_1,\ldots,c_m\rangle)|\\
        &\leq \alpha \rho^n \mu_d(\langle b_1,\ldots,b_k \rangle)\mu_d (\langle c_1,\ldots,c_m\rangle).
    \end{split}
    \end{equation}
\end{proposition}
We define a piece of notation which will give a simple interpretation to this proposition. If $Q(a_1(z),a_2(z),\ldots)$ is a set of conditions concerning the digits of the continued fraction expansion, we write 
\[\mathbb{P}(Q(a_1(z),a_2(z),\ldots)) = \mu_d\{z\in K_d:Q(a_1(z),a_2(z),\ldots) \text{ holds }\}.\]
We may then rewrite $\mu_d(\langle b_1,b_2,\ldots,b_k \rangle \cap T^{-(n+k)} \langle c_1,\ldots,c_m\rangle) $ in \eqref{equation: mixing map} as 
\[\mathbb{P}(a_1(z)=b_1,\ldots ,a_k(z) = b_k,\quad a_{k+n+1}(z) = c_1,\ldots ,a_{k+n+m+1}(z)  = c_m).\]
Similarly, the quantities $\mu_d(\langle b_1,\ldots,b_k \rangle)$ and $\mu_d( \langle c_1,\ldots,c_m\rangle)$ can be thought of as $\mathbb{P}(a_1(z)=b_1,\ldots ,a_k(z) = b_k)$ and $\mathbb{P}(a_{1}(z) = c_1,\ldots ,a_{m}(z)  = c_m)$ respectively.  Equation \eqref{equation: mixing map} therefore tells us that finite sequences of digits occur almost independently with an error term that decreases exponentially with respect to the gap between the sequences. If we take $\xi_l=a_l$ in Definition \ref{definition: psi-mixing}, we note that the $\sigma$-algebra $\mathcal{F}_{l}^m$ is generated by the level sets $\{z:a_n(z)=b\}$ for all $b\in\mathfrak{o}_d, b\leq n\leq m$. Hence, Proposition \ref{proposition: exponential decay} is equivalent to '$\psi$-mixing' with exponentially decaying $\alpha_n$. For $\psi$-mixing sequences, a Poisson law follows from \cite{Iosifescu}. 
\begin{proposition}[Iosifescu]\label{proposition: LLR}
    Let $\xi_1,\xi_2,\ldots$ be a sequence of identically distributed $\psi$-mixing random variables. Moreover, assume that there is some sequence $(u_n)_n$ with $u_n\to\infty$ as $n\to\infty$ that satisfies \begin{equation}
        \limsup_{n\to\infty} n\mathbb{P}(\xi_1>u_n) = \tau.
    \end{equation}
    for some $\tau\in\mathbb{R}$. Then the $l$-th largest value $M^{(l)}_n$ in $\{\xi_1,\xi_2,\ldots ,\xi_n\}$ satisfies $\mathbb{P}( M^{(l)}_n \leq u_n)\to  e^{-\tau}\sum_{j=0}^{l-1}\tau^j/j!$ as $n\to\infty$.
\end{proposition}

\begin{proof}[Proof of Theorem \ref{theorem: Gumbel Law for Bianchi Groups} for $\eta = \mu_d$]
    We show that the theorem holds for the measure $\mu_d$ when $C_d=\sqrt{H}$, we can apply Proposition \ref{proposition: LLR}, which states that it suffices to show that 
    \[n\mu_d\left\{z\in K_d: \abs{a_n(z)}> u\sqrt{nH } \right\}\to \frac{1}{u^2}\text{ as } n\to\infty.\]
    This follows immediately from Lemma \ref{lemma: chance of large digits}. This proves the theorem for $\eta = \mu_d$. Extending this to any measure equivalent to the Lebesgue measure is a straightforward application of a result due to Eagleson, which we state below.
\end{proof}

Let $(R_n)_n$ be a sequence of random variables on some probability space $(X,\mathcal{A},P)$ and let $R$ be another random variable, not necessarily defined on $X$. We write $R_n\xRightarrow{P} R $ if $R_n$ converges to $R$ in distribution, i.e. $P\circ R_n^{-1}$ converges weakly to the law of $R$ as $n\to\infty$. We use the following theorem due to Eagleson, see \cite{Eagleson}, \cite{Zweimuller}. 
\begin{proposition}[Eagleson]\label{proposition: Eagleson}
    Let $(R_n)_n$ be a sequence of random variables on some probability space $(X,\mathcal{A},P)$. Let $R$ be another random variable and assume $R_n\xRightarrow{P} R $. Assume there exists a $\sigma$-finite measure $\mu$ such that $P$ is absolutely continuous with respect to $\mu$, and an ergodic nonsingular map $M$ on $(X,\mathcal{A},\mu)$ for which
    \[\abs{R_n\circ M-R_n}\xRightarrow{\mu} 0.\]
    Then for all measures $\eta$ which are absolutely continuous with respect to $\mu$, we have that 
    \[R_n \xRightarrow{\eta} R.\]
\end{proposition}
\begin{proof}[Proof of Theorem \ref{theorem: Gumbel Law for Bianchi Groups} for $\eta$ equivalent to the Lebesgue measure.]
Let $R$ be some $(0,\infty)$-valued random variable with cumulative distribution function $F_R(u) = e^{-1/u^2}$. Theorem \ref{theorem: Gumbel Law for Bianchi Groups} for the measure $\mu_d$ therefore states that
\[R_N:=\frac{\max_{n\leq N}\abs{a_n(z)}}{C_d\sqrt{N}}\xRightarrow{\mu_d} R \text{ as }N\to\infty.\]
Since $a_n(G_d(z))= a_{n+1}(z)$ for all $n\in\mathbb{N}$, all $u\geq0$ satisfy
\[ \mu_d\left\{\beta:\abs{R_n\circ G_d-R_n}(z)\leq u\right\} \leq \mu_d\left\{\beta:\frac{\max\{\abs{a_1(z)},\abs{a_{N+1}(z)}\}}{C_d\sqrt{N}}\leq u\right\} .\]
By Proposition \ref{proposition: exponential decay}, this tends to zero as $N\to \infty$ if $u>0$. This proves $\abs{R_n\circ G_d-R_n}\xRightarrow{\mu_d} 0$, which means that we can apply Proposition \ref{proposition: Eagleson} and obtain that 
\begin{equation}\label{equation: Eagleson for Galambos Result}
    R_N\xRightarrow{\eta} R,
\end{equation}
for any $\eta$ which is absolutely continuous with respect to $\mu_d$, or equivalently, the Lebesgue measure.
\end{proof}

\section{Bianchi orbifolds}\label{section: Bianchi Orbifolds}
This section is dedicated to elucidating the geometry of the five Bianchi orbifolds under consideration. Let us start with a short description of the construction of the three dimensional hyperbolic space. We refer to \cite{GroupsOnH3} for more information. Consider the subset of the quaternions defined by 
\[\mathbf{H}^3 = \{x+yi+rj:x,y\in\mathbb{R},r>0\}.\]
With the Riemannian metric
\begin{equation}\label{equation: Riemannian metric}
    ds^2 = \frac{dx^2+dy^2+dr^2}{r^2},
\end{equation}
it is the unique simply connected three dimensional Riemannian manifold of constant negative sectional curvature $-1$. Let $I$ be the $2\times2$ identity matrix. For $w\in \mathbf{H}^3$, there is a unique $z\in\mathbb{C}$ and $r>0$ such that $w=z+rj$. In these coordinates, we can write the action of $\PSLC=\SLC/\{\pm I\}$ on $\mathbf{H}^3$ as
\begin{equation}\label{equation: Mobius transformation in coordinates}
\begin{split}
        \begin{pmatrix} a&b \\ c&d \end{pmatrix}\cdot (z+rj)&= (aw+b)(cw+d)^{-1}\\
        &=\frac{(az+b)(\overline{c}\overline{z}+\overline{d})+a\overline{c}r^2}{|cz+d|^2+|c|^2r^2}+\frac{rj}{|cz+d|^2+|c|^2r^2},
\end{split}
\end{equation}
where the right hand side of the first equality should be interpreted as an expression in the algebra of quaternions. The boundary $\partial \mathbf{H}^3$ of $\mathbf{H}^3$ is the Riemann sphere $\widehat{\mathbb{C}} =\mathbb{C}\cup\{\infty\} $. The action \eqref{equation: Mobius transformation in coordinates} extends to the familiar action on $\partial \mathbf{H}^3$ by Möbius transformations by setting $r=0$ in \eqref{equation: Mobius transformation in coordinates}.

The following proposition lists a complete set of generators for the Bianchi groups under consideration. This result can be found in \cite{Cohn68}.
\begin{proposition}\label{proposition: complete set of generators}
    Let $\omega_d = \frac{(-1+\sqrt{-d})}{2}$ if $d= 3 \mod 4$, and let $\omega_d = \sqrt{-d}$ otherwise. Let $\overline{\omega_d}$ be the complex conjugate of $\omega_d$. Define
    \[S=\begin{pmatrix}
            0 & 1\\
            -1& 0
        \end{pmatrix}\text{ and } T^q =\begin{pmatrix}
            1 & q\\
            0 & 1
        \end{pmatrix} \]
    for $q\in \mathfrak{o}_d$.  If $d= 2,7$, or $11$, then 
    \begin{equation}\label{equation: generators for Bianchi groups}
    S,T^1\text{ and }T^{\omega_d}
    \end{equation}
    generate $\PSLo$. If $d=1,3$, then we must add 
    \[\begin{pmatrix}
            \omega_d & 0\\
            0 & \overline{\omega_d} 
    \end{pmatrix}\]
    to \eqref{equation: generators for Bianchi groups} to form a complete set of generators. 
\end{proposition}
Geometrically, the map $S$ is an inversion with respect to the unit hemisphere $\{x+yi+rj\in\mathbf{H}^3:x^2+y^2+r^2=1\}$ followed by a reflection across the $yr$-plane. The map $T^q$ is the translation $z+rj\mapsto z+q+rj$ parallel to the complex plane. 

The five Bianchi groups we consider have a fundamental domain that is simple to describe. 
\begin{proposition}\label{proposition: Fundamental Domain for Bianchi Groups}
    If $d=2,7$ or $11$, a fundamental domain for $\PSLo$ is given by 
    \[\mathcal{F}_d = \{z+rj\in\mathbf{H}^3: z\in K_d, \abs{z}^2+r^2>1\}.\]
    For $d=1,3$, the above domain must be modified to account for the extra generator in Proposition \ref{proposition: complete set of generators}. Define 
    \[K_1' = \left\{x+iy:x\in\mathbb{C}: 0<x<\frac{1}{2}, |y|< \frac{1}{2} \right\} \]
    and
    \[K_3' = \left\{z\in K_3: \arg{z}\in \left[\frac{\pi}{6},\frac{5\pi}{6}\right] \right\}.\]
    The fundamental domain is given by 
    \[\mathcal{F}_d = \{z+rj\in\mathbf{H}^3: z\in K_d', \abs{z}^2+r^2>1\}\]
    for $d=1,3$.
\end{proposition}
The fundamental domain for $d=2$ is depicted in Figure \ref{fig:F2Bianchi}.
\begin{proof}[Proof of Proposition \ref{proposition: Fundamental Domain for Bianchi Groups}]
    These results date back to Bianchi \cite{Bianchi1892}. An exposition can be found in \cite{SWAN19711}, Chapters 5, 6, 10, 13.
\end{proof}
\begin{figure}
    \centering
    \includegraphics[width=\textwidth]{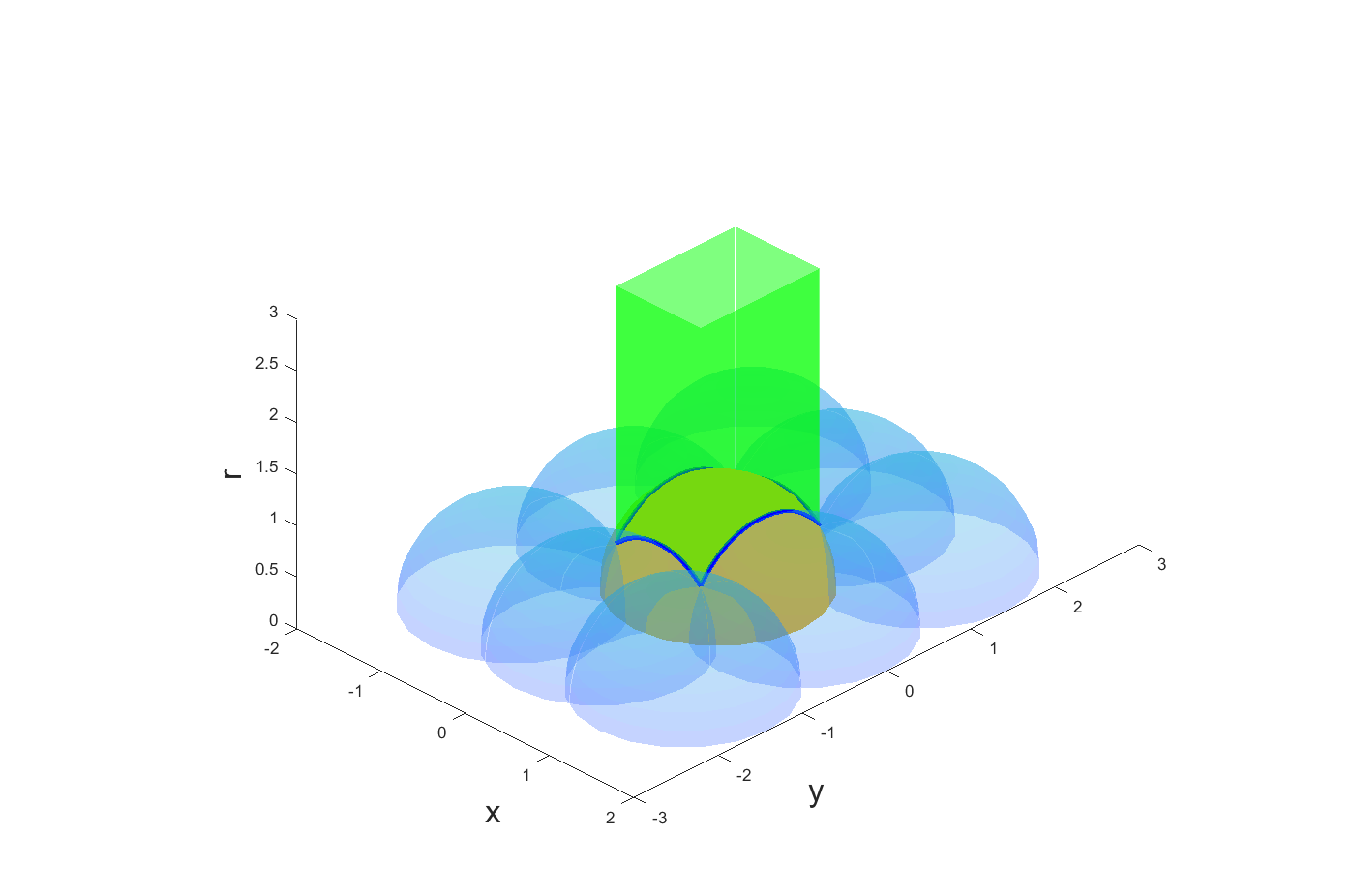}
    \caption{Construction of the fundamental domain for $d=2$. The points $x+yi+rj\in \mathcal{F}_2$ belong to the region bounded by the opaque hemisphere and the four vertical rectangles.}
    \label{fig:F2Bianchi}
\end{figure}
The orbifold $\mathbf{P}_d$ is noncompact and unbounded with respect to the hyperbolic metric. The unbounded end of $\mathbf{P}_d$ corresponding to $r\to \infty$ in $\mathcal{F}_d$ is called the \textbf{cusp}.

In what follows, when we say something holds almost everywhere, we mean with respect to either the Liouville measure on the relevant unit tangent bundle, or the Lebesgue measure on $K_d$.   

\section{From a Fréchet law to cusp excursions}\label{section: cusp excursions}
In \cite{Pollicott2007LIMITINGDF} the geodesic flow on $\mathbf{M}$ is modeled as a suspension flow over the natural extension of the real-valued Gauss map. As far as we are aware, there is no such known interpretation for nearest-integer $\mathfrak{o}_d$-continued fractions. Nicolas Chevallier obtains such a result for a different Gaussian continued fraction algorithm in \cite{chevallier2021gauss}. Lukyanenko and Vandehey \cite{Lukyanenko} establish a partial result by considering so-called geodesic marking for various continued fraction expansions associated to a large class of hyperbolic orbifolds. 

Nevertheless, we are able to establish a weaker correspondence between the digits of the continued fraction expansion and the height of these cusp excursions.

Let us first establish some geometrical intuition. In what follows let $H(b)\subset \mathbf{H}^3$ be the hemisphere of Euclidean radius one with centre $b \in\mathbb{C}$. Suppose we have a geodesic $\gamma$ on $\UT\mathbf{P}_d$. Let $\tilde{\gamma}_0$ be a lift of this geodesic to $\mathbf{H}^3$. We say that $\tilde{\gamma}_0$ has endpoints $(z_1,z_2)$ if $z_1,z_2\in \partial\mathbf{H}^3$ are endpoints of $\tilde{\gamma}_0$ and $z_1$ is the attracting endpoint. Let us assume that $\tilde{\gamma}_0$ satisfies $z_1\in K_d$ and $|z_2|> 1$ \footnote{It is in fact always possible to find such a lift \cite{Abrams}.}. Let $z_1 = [0,a_1,a_2,\ldots]$ and let us first assume for simplicity that this expansion contains no small digits, i.e. $\abs{a_n}\geq 3$ for all $n$. This geodesic intersects $H(0)$ in $\tilde{\gamma}_0(t_0)$ at some time $t_0$. By applying the map $T^{a_1}S$, we obtain another lift $\tilde{\gamma}_1$ of $\gamma$ with endpoints 
\[\left(z_1^{(1)},z_2^{(1)}\right) = -\left(\frac{1}{z_1}-a_1, \frac{1}{z_2}-a_1\right).\]
These endpoints once again satisfy the condition $z_1^{(1)}\in K_d$ and $\abs{z_2^{(1)}}>1$. Define $t_1$ to be the time at which $\tilde{\gamma}_1$ intersects $H(0)$.  
The map $S$ maps the interior of the unit hemisphere $H(0)$ to its exterior and vice versa. Hence the map $T^{a_1}S$ maps the exterior of $H(0)$ to the interior of $H(a_1)$ and maps $\tilde{\gamma}_0(t_0)\in H(0)$ to $\tilde{\gamma}_1(t_0)\in H(a_1)$.
\begin{figure}
    \centering
    \includegraphics[width=0.7\textwidth]{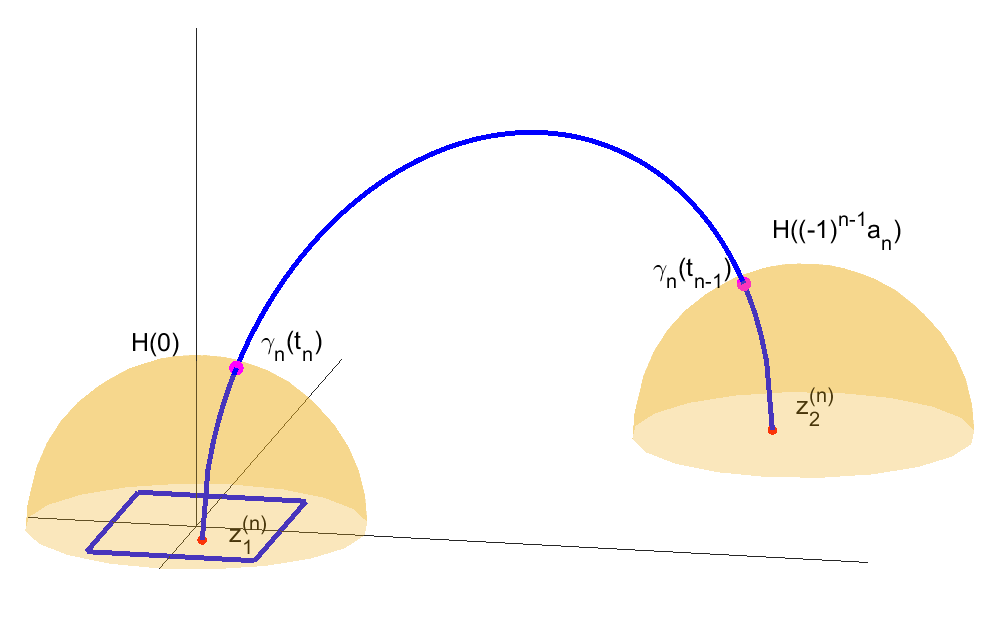}
    \caption{The geodesic $\tilde{\gamma}_n$ on $\mathbf{H}^3$ which arcs from the point $z_1^{(n)}$ to $z_2^{(n)}$. It intersects the spheres $H(a_n)$ and $H(0)$ at time $t_{n-1}$ and $t_{n}$ respectively. }
\label{figure: geodesic arcing between spheres}
\end{figure}
The time interval $[t_0,t_1]$ therefore corresponds to the time $\tilde{\gamma}_1$ spends above the two hemispheres. Take $n=1$ in Figure
\ref{figure: geodesic arcing between spheres} for an illustration. By symmetry of the choice function $\ch{\cdot}_d$, we have that $z_1^{(1)}= [0,-a_2,-a_3,\ldots]$ for almost all values of $z_1$. We can therefore apply the map $T^{-a_2}S$ to obtain a new geodesic lift $\tilde{\gamma}_2$. Iterating this process, we apply the maps $T^{a_3}S,T^{-a_4}S,\ldots$ in order to obtain a sequence of lifts $\{\tilde{\gamma}_n\}_n$ of $\gamma$. Define the sequence $\{t_n\}_n$ by the relation $\tilde{\gamma}_n(t_n)\in H(0)$. Each interval $[t_{n-1},t_n]$ corresponds to the segment of $\tilde{\gamma}_n$ above the hemispheres $H(0)$ and $H\left((-1)^{n-1}a_n\right)$. We refer once again to Figure \ref{figure: geodesic arcing between spheres}. By definition of the metric \eqref{equation: Riemannian metric}, the distance between this apex and the horocycle $\{z+j:z\in \mathbb{C}\}$ is the logarithm of the Euclidean height of the apex, where we define the Euclidean height of a point $z+rj\in \mathbb{H}^3$ to be equal to $r$. Understanding the long term behavior of $\gamma$ therefore boils down understanding the asymptotic behavior of $t_n$ and the height of the apexes. 

The assumption that there are no small digits is too restrictive, however. Indeed, a straightforward consequence of the existence of a Lebesgue-equivalent ergodic invariant probability measure is that the $\mathfrak{o}_d$-continued fraction expansions of almost all complex numbers contain small digits. We discuss these technicalities in the next subsection. 

\subsection{Excursion times}\label{subsection: excursion times}
One complication that may arise with small digits is that the repelling endpoints do not necessarily satisfy the condition $\abs{z_2^{(n)}}>1$ and hence the times $t_n$ discussed in Section \ref{section: cusp excursions} are not always defined. 
To avoid this, we restrict to geodesics $\gamma$ on $\UT\mathbf{P}_d$ with a lift $\tilde{\gamma}_0$ where the repelling endpoint is $\infty$. If such a lift exists, we may assume its endpoints are $(\beta,\infty)$ for some $\beta\in K_d$. In the final proof of Theorem \ref{theorem: cusp excursions on Euclidean Bianchi Orbifolds}, we argue that for sufficiently large times $t$, generic geodesics exhibit the same statistical properties as these model geodesics.

We can also reparametrise $\gamma, \tilde{\gamma}_0$ so that $\tilde{\gamma}_0(0)= \beta+j$. As before, we define
\begin{equation}\label{equation: definition of tilde gammas}
    \tilde{\gamma}_n := (T^{(-1)^{n-1}a_n}S)\circ\tilde{\gamma}_{n-1}
\end{equation}
for all $n\in\mathbb{N}$.
For almost all values of $\beta$, the endpoints of $\tilde{\gamma}_n$ are given by  
\[(-1)^n\hat{G}_d^n(\beta, \infty),\]
where $\hat{G}_d$ is the extension of the map $G_d$ defined by
\[\hat{G}_d(z,w) = \left(\frac{1}{z}-\ch{\frac{1}{z}}_d,\frac{1}{w}-\ch{\frac{1}{z}}_d\right)\]
for all $z,w\in\widehat{\mathbb{C}}$ with $z\neq w$. If $\beta = [0,a_1,a_2,\ldots]$ and $w=\infty$, we see that 
\begin{equation}\label{equation: iterated Gauss}
    \hat{G}^n_d(\beta, \infty) = \left(G_d^n(\beta),-a_n+\frac{1}{ -a_{n-1}+ \frac{1}{\ddots -a_2+\frac{1}{-a_1}}}\right)=\left(G_d^n(\beta),\frac{-q_n(\beta)}{q_{n-1}(\beta)}\right),
\end{equation}
where the last equality is a consequence of \eqref{equation: reversing the digits}. Since $\abs{\frac{q_n}{q_{n-1}}}>1$, the following is well-defined.
\begin{definition}\label{definition: intersection time}
    Let $\beta=[0,a_1,a_2,\ldots]\in K_d$ and let $\{\tilde{\gamma}_n\}_{n\in\mathbb{N}}$ be defined  as above. We define the \textbf{intersection times} $\{t_n\}_n$ by setting $t_n$ to be the time the geodesic $\tilde{\gamma}_n$ intersects with the unit hemisphere $H(0)= \{z+rj\in \mathbf{H}^3:\abs{z}^2+r^2=1\}$ centred at the origin.
\end{definition}

A problem may occur with small digits, however. Indeed, if $\abs{a_n}\leq 2$, then $H(0)$ and $H((-1)^{n-1}a_n)$ overlap. In particular, it is possible $\tilde{\gamma}_n$ intersects the unit hemisphere $H(0)$ first, which would imply that that $t_{n-1}>t_{n}$. Nonetheless, we obtain the following result.

\begin{lemma}\label{lemma: bounding the intersection time}
     Let $(p_n/q_n)_n$ be the convergents of $\beta$ with respect to the nearest-integer $\mathfrak{o}_d$-continued fraction expansions. For all $n\geq1$,
    \begin{equation}\label{equation: intersection time estimate}
        \abs{t_n - 2\log \abs{q_n} - \frac{3}{2}\log\left(1-\abs{\frac{q_{n-1}}{q_n}}^2\right)}\leq D,
    \end{equation}
    for some constant $D$ which depends only on $d$. 
\end{lemma}
The proof involves finding the matrix  $g\in \PSLo$ such that $\tilde{\gamma}_n =g\tilde{\gamma}_0$ and representing it in terms of the associated quotient pair of sequences. By a straightforward but somewhat lengthy calculation, we find an explicit formula for $t_n$ from which the lemma follows. The details can be found in Section \ref{section: intersection times and cusp excursion}. 

The following definition provides a monotone sequence based on the intersection times.
\begin{definition}\label{definition: excursion times}
    Define the \textbf{excursion times} $(t^*_n)_{n\in\mathbb{N}}$ by setting 
    \[t^*_n = \max_{i\in\{1,2,\ldots n\}} t_i. \]
\end{definition}

In Section \ref{section: intersection times and cusp excursion}, we use standard results in the theory of continued fractions to argue that there is some $r>1$ such that for all irrational numbers $z\in K_d$, there is a subsequence $q_{n_k}$ of denominators such that $\abs{q_{n_k+1}/q_{n_k}}>r$ for all $k$. Using Lemma \ref{lemma: bounding the intersection time} and an application of the Birkhoff ergodic theorem, we shall obtain the following proposition.
\begin{proposition}\label{proposition: asymptotic formula for excursion times}
There exists a constant $C^*>0$ such that for almost every $\beta \in K_d$, the excursion times $(t_n^*)_n$ in Definition \ref{definition: excursion times} satisfy 
\[\lim_{n\to \infty} \frac{t_n^*}{n} = C^*.\]
\end{proposition}
The details may be found in Section \ref{section: intersection times and cusp excursion}. 

We establish a correspondence between the digits of the continued fraction expansion and the maximum height of the corresponding cusp excursions, which we measure as follows. Denote by $\mathcal{H}$ the natural projection of the horosphere $\{z+j:z\in \mathbb{C}\}\subset \mathbb{H}^3$ to $\mathbf{P}_d$. Then we define $h: \UT\mathbf{P}_d\to [0,1) $ to be the function which assigns to a vector $w\in \UT\mathbf{P}_d$ the hyperbolic distance in $\mathbf{P}_d$ from the base point of $w$ to $\mathcal{H}$. 
\begin{lemma}\label{lemma: excursion height is given by digits}
    Let $\tilde{w}_\beta$ be the tangent vector to $\tilde{\gamma}_0$ at $\tilde{\gamma}_0(0)$. Let $\tilde{w}_\beta(t)$ be the vector obtained by applying the geodesic flow for some time $t$. Let $w_\beta$ and $w_\beta(t)$ be their respective projections to $\UT \mathbf{P}_d$. The sequence
    \[n\mapsto \abs{\max_{t\in [0,t_n^*]} e^{h(w_\beta(t))}- \frac{1}{2}\max_{k\leq n}\abs{a_k}}\]
    is uniformly bounded from above for all $\beta\in K_d\backslash \mathbb{Q}$.
\end{lemma}
By \eqref{equation: iterated Gauss} we have that the Euclidean height of the apex of $\tilde{\gamma}_n$ is approximately equal to $|a_n|/2$. The hyperbolic distance to the horocycle $\{z+j:z\in\mathbb{C}\}$ is the logarithm of that. If $\abs{a_n}>2$, then $\tilde{\gamma}_n$ reaches this apex in the time interval $[t_{n-1},t_{n}]$. The remainder of the argument involves proving the intuitive assumption that sequences of small digits do not correspond to cusp excursions. A detailed proof can be found in Section \ref{section: intersection times and cusp excursion}.

\subsection{Proof of Theorem \ref{theorem: cusp excursions on Euclidean Bianchi Orbifolds}}\label{subsection: main theorem}
Theorem \ref{theorem: cusp excursions on Euclidean Bianchi Orbifolds} is a theorem about the Liouville measure $m$ on $\UT\mathbf{P}_d$. We define a surjective map from $\UT\mathbf{P}_d$ to $K_d$ for which the pushforward of $m$ is absolutely continuous with respect to the Lebesgue measure. This will allow us to apply Theorem \ref{theorem: Gumbel Law for Bianchi Groups}. Recall that we defined $\mathcal{F}_d$ to be the fundamental domain of the action of $\PSLo$ on $\mathbf{H}^3$. For almost all $v$, there is a unique lift $v'\in \UT\mathbb{H}^3$ with base point in $\mathcal{F}_d$. Almost surely, there exists a unique $q\in\mathfrak{o}_d$ such that if we define
\[\tilde{v}:= T^q v'=\begin{pmatrix}
            1 & q\\
            0 & 1
        \end{pmatrix} v',\] 
then the oriented unit speed geodesic $t\mapsto\tilde{v}(t)$ with tangent vector $\tilde{v}$ at time $t=0$ has attracting endpoint $\tilde{v}_\infty\in K_d$. Let $\pi: \UT\mathbf{P}_d\to K_d$ be the map which sends $v\mapsto \tilde{v}_\infty$.
The following result is easily seen. 
\begin{lemma}\label{lemma: pushforward is absolutely continuous}
    Let $m$ be the Liouville measure on $\UT \mathbf{P}_d$. The pushforward measure $m\circ \pi^{-1}$ on $K_d$ is equivalent to the Lebesgue measure. 
\end{lemma}

\begin{proof}[Proof of Theorem \ref{theorem: cusp excursions on Euclidean Bianchi Orbifolds}]
    We first claim that elements in the same fibre of $\pi$ have the same asymptotic behaviour for large $T$. Let $v\in\pi^{-1}(\beta)$ and let $\tilde{w}_\beta, w_\beta$ be as in Lemma \ref{lemma: excursion height is given by digits}.
    By definition of $\pi$, there is a lift $\tilde{v}\in \UT\mathbf{H}^3$ of $v$ with attracting endpoint $\beta$. Elementary hyperbolic geometry shows that there exists a $\Delta\in \mathbb{R}$ such that $d(\tilde{w}_\beta(t+\Delta),\tilde{v}(t)) \to 0\text{ as }t\to\infty$. Hence $d(w_\beta(t+\Delta),v(t)) \to 0\text{ as }t\to\infty$, which proves the claim. Every $\alpha\in \mathbb{R}$ therefore satisfies
    \begin{equation}\label{equation: cuspexcursionmanipulations1}
        \begin{split}
            &\lim_{T\to\infty} m\left\{ v\in \UT\mathbf{P}_d : \sup_{0 \leq t\leq T} h(v(t)) -\frac{1}{2}\log T \leq \log(u)+\alpha\right\}\\
            =& \lim_{T\to\infty} (m\circ\pi^{-1})\left\{ \beta\in K_d : \sup_{0 \leq t\leq T} h(w_\beta(t)) -\frac{1}{2}\log T \leq \log(u)+\alpha\right\}\\
            =& \lim_{T\to\infty} (m\circ\pi^{-1})\left\{ \beta\in K_d : \sup_{0 \leq t\leq T} e^{h(w_\beta(t))}\leq e^\alpha \sqrt{T}u\right\},
        \end{split}
    \end{equation}
 provided the first limit exists. Let $(t_n^*)_n$ as in Definition \ref{definition: excursion times}. Proposition \ref{proposition: asymptotic formula for excursion times} states that the typical number of cusp excursions $N$ after some time $T$ is approximately equal to $T/C^*$. Lemma \ref{lemma: excursion height is given by digits} shows that the maximum value $e^{d(w_\beta,w_\beta(t))}$ attained in these $N$ excursions is asymptotically equal to $\frac{1}{2}\max_{1\leq n\leq N}|a_n(\beta)|$, so we obtain that the expression above is equal to 
\begin{equation}\label{equation: cuspexcursionmanipulations2}
    \lim_{N\to +\infty} (m\circ\pi^{-1})\left\{ \beta\in K_d : \max_{1\leq n\leq N}|a_n(\beta)|\leq 2e^\alpha \sqrt{C^* N}u\right\}.
\end{equation}
    A more precise derivation of the above step is given in Section \ref{section: intersection times and cusp excursion}.
    Define $\kappa_d$ to be the solution to $2e^\alpha \sqrt{C^*} =C_d$, where $C_d$ is the constant appearing in Theorem \ref{theorem: Gumbel Law for Bianchi Groups}. If $\alpha = \kappa_d$ then by Proposition \ref{proposition: limit theorem for complex continued fractions but complicated} the limit evaluates to $e^{-1/u^2}$, which is what we had to prove. A version of this theorem for the $k$-th largest excursion follows analogously after replacing $e^{-1/u^2}$ with $e^{-1/u^2}\sum_{l=0}^{k-1} u^{-2l}/l^2$.
\end{proof}

\section{From cusp excursions to Diophantine approximation}\label{section: Diophantine}
The proof of Theorem \ref{theorem: cusp excursions on Euclidean Bianchi Orbifolds} does not give an explicit expression for the constant $\kappa_d$, since we do not have an exact form for the invariant measure $\mu_d$, see \cite{Ei2019}, \cite{Ei23+}. We obtain this constant by comparing results from our analysis of nearest-integer continued fraction expansions with those obtained from a version of the Koebe-Morse coding for three-dimensional orbifolds. This involves calculating some volumes, so we first briefly discuss some natural measures on $\mathbf{H}^3$ and $\UT\mathbf{H}^3$. 

\subsection{The Liouville measure on $\UT\mathbf{H}^3$}
Recall that the natural volume measure $m_{\mathbb{H}^3}$ induced by the Riemannian Metric \eqref{equation: Riemannian metric} on $\mathbb{H}^3$ is given by \[dm_{\mathbb{H}^3}=\frac{dxdydr}{r^3}.\] 
Denote the $2$-sphere by $\mathbb{S}^2$. Under the natural identification of $\UT\mathbb{H}^3$ with $\mathbb{H}^3\times \mathbb{S}^2$ the Liouville measure $m$ is given by $m= \mathrm{vol}\times m_{\mathbb{S}^2}$, where $m_{\mathbb{S}^2}$ is the natural measure on the sphere. It will be useful to consider the following alternative coordinate system on $\UT\mathbf{H}^3$.

\begin{lemma}\label{lemma: Endpoint Coordinates for the Unit Tangent Bundle}
    For a given unit tangent vector $(w,u)$, let $(\beta,\alpha)$ the endpoints of the unique geodesic $\tilde{\gamma}$ it is on. Reparametrise the geodesic so that $\tilde{\gamma}(0)$ is the Euclidean apex of the geodesic and let $t\in\mathbb{R}$ satisfy $(\tilde{\gamma}(t).\tilde{\gamma}'(t))= (w,u)$. In these coordinates, the Liouville measure is given by 
\begin{equation}\label{equation: Liouville in alpha-beta coordinates}
    dm = \frac{4d\lambda(\alpha)d\lambda(\beta)dt}{\abs{\beta-\alpha}^4},
\end{equation}
where $\lambda$ is the Lebesgue measure on $\mathbb{C}$.
\end{lemma}
\begin{remark}
    It is straightforward to verify that this measure is indeed invariant under all Möbius transformations and therefore indeed corresponds to the Liouville measure on $\UT\mathbb{H}^3$.
\end{remark}

\subsection{Calculating the constant in Theorem \ref{theorem: Gumbel Law for Bianchi Groups}.}\label{subsection: calculating alpha}
We introduced some unknown constants in the proof of Theorem \ref{theorem: Gumbel Law for Bianchi Groups}. However,we shall show it is fairly straightforward application of Kac's Lemma to calculate the probability that a given vector in $\mathbf{P}_d$ sampled from the Liouville measure lies a certain distance up the cusp, and use the relationship between cusp excursions and continued fractions to calculate $\kappa_d$. 
 Consider first the case where $d=2,5,7,11$. The cases $d=1,3$ require minor modifications due to the extra generator for the Bianchi group.

\begin{lemma}\label{lemma: average return times to to A_R}
    Use $(\beta,\alpha,t)$-coordinates from Lemma \ref{lemma: Endpoint Coordinates for the Unit Tangent Bundle}, and let $W:(\mathbb{C}\backslash \{(z,z):z\in\mathbb{C}\})\times \mathbb{R}\to\UT\mathbb{H}^3$ be the coordinate map which sends $(\beta,\alpha,t)$ to its corresponding unit vector. Consider the set 
    \begin{equation}\label{equation: tilde A_R}
       \tilde{A}_R=\{ W(\beta,\alpha,0) \in\UT\mathbf{H}^3: \beta\in K_d, \abs{\alpha}> R\},
    \end{equation}
     where $R>3$. Note that this is a collection of horizontal unit tangent vectors at height $\abs{\beta-\alpha}/2$. Denote by $A_R$ its projection to $\UT\mathbf{P}_d$. Then for almost all $v$,
     \begin{equation}
    \lim_{T\to +\infty}\frac{ \#\{t\in[0,T]: v(t)\in A_{R} \} }{T} = 4\frac{\pi \lambda(K_d)}{\text{vol}(\mathbf{P}_d)R^2}+O\left(\frac{1}{R^3}\right),
\end{equation}
where $\text{vol}$ is the hyperbolic volume measure.
\end{lemma}
\begin{proof}
    We use a thickening argument. Given a tangent vector $v\in\UT\mathbf{H}^3$ or $\UT\mathbf{P}_d$, denote by $v(t)$ the point obtained by applying the geodesic flow for time $t$. The Liouville measure $m$ restricts to a measure on $\UT\mathbf{P}_d$ by identifying the latter with its fundamental domain $\mathcal{F}_d$. We will also denote this measure by $m$.
    
    Let $\tilde{A}_{R,\epsilon}=\{\tilde{v}(t):\tilde{v}\in \tilde{A}_R, t\in[0,\epsilon]\}$ with $\epsilon>0$. For sufficiently small $\epsilon$, denote by $A_{R,\epsilon}$ the projection to $\UT\mathbf{P}_d$. Note that $A_{R,\epsilon} = \{v(t):v\in A_R, t\in[0,\epsilon]\}$. For sufficiently small $\epsilon$, we claim that $m(A_{R,\epsilon} ) = m(\tilde{A}_{R,\epsilon})$. Indeed, vectors lying above the unit hemispheres $\{z+rj:\abs{z-a}^2+r^2=1\}$ for all $a\in\mathfrak{o}_d$ can only be identified by translations in $\mathfrak{o}_d$. By our assumption on $R$, the unit tangent vectors lie above the horocycle $r=2$ and thus above the hemispheres, hence it follows from the definition of $A_R$ and $A_{R,\epsilon}$ that the projection map restricted to $A_{R,\epsilon}$ is injective.
    
Fix $R>3$. By ergodicity of the geodesic flow and Kac's lemma, we have for almost all $v\in\UT\mathbf{P}_d$ that
\begin{equation}
    \lim_{T\to +\infty} \int_{0}^T \mathbf{1}_{A_{R,\epsilon}}(v(t)) = \frac{m(\tilde{A}_{R,\epsilon})}{\text{vol}(\mathbf{P}_d)},
\end{equation}
where $\mathbf{1}_{A_{R,\epsilon}}$ is the characteristic function on $A_{R,\epsilon}$. 
In $(\beta,\alpha,t)$, coordinates, we note that 
\[m(\tilde{A}_{R,\epsilon}) = \epsilon\int_{K_d}\int_{\alpha>R} \frac{4}{\abs{z-w}^4}d\lambda(\alpha)d\lambda(\beta).\]
On each occasion in which $v(t)$ enters $A_{R,\epsilon}$, it enters through $A_R$ and spends a time of $\epsilon$ in $A_{R,\epsilon}$ before leaving the set. We thus obtain for almost all $v\in\UT\mathbf{P}_d$ that 

\begin{equation}\label{equation: continuity of peaking times}
    \lim_{T\to +\infty}\frac{ \#\{t\in[0,T]: v(t)\in A_{R} \} }{T} = \frac{1}{\text{vol}(\mathbf{P}_d)} \int_{K_d}\int_{\alpha>R} \frac{4}{\abs{\alpha-\beta}^4}d\lambda(z)d\lambda(\beta).
\end{equation}
Since a countable intersection of sets of full measure is also of full measure, we have that almost all $v\in\UT\mathbf{P}_d$ satisfy the above formula for all $R\in\Big(\mathbb{Q}\cap [3,\infty)\Big)$. By continuity in $R$ on the right hand side and monotonicity on both sides, this also holds for all $R\in[3,\infty)$.

Since $|z-w|^{-4} = |z|^{-4}+O(\abs{z}^{-5})$ as $\abs{z}\to +\infty$ for $w\in K_d$, we obtain that 
\begin{equation}
    \begin{split}
        \lim_{T\to +\infty}\frac{ \#\{t\in[0,T]: v(t)\in A_{R} \} }{T} &= \frac{1}{\text{vol}(\mathbf{P}_d)} \int_{K_d}\int_{\alpha>R} \frac{4}{\abs{\alpha}^4}d\lambda(\alpha)d\lambda(\beta)+O(r^3)\\
        &=\frac{2\pi\lambda({K}_d)}{\text{vol}(\mathbf{P}_d)} \int_R^{+\infty} \frac{4}{\tau^3}d\tau+O(r^3);
    \end{split}
\end{equation}
the lemma follows. 
\end{proof}

Consider once more the notation of Lemma \ref{section: cusp excursions}.  In particular, consider the model geodesic $\gamma$, its lifts $\tilde{\gamma}_n$ and the model tangent vector $w_\beta$ from Subsection \ref{subsection: main theorem}. As in Section \ref{section: cusp excursions}, denote the tangent vector to $\gamma$ at time $t$ by $w_\beta(t)$. In the proof of Lemma \ref{lemma: excursion height is given by digits} in Section \ref{section: intersection times and cusp excursion}, we show that if $\beta \notin \mathbb{Q}$ there exists some sequence $m_k$ of increasing natural numbers for which $\tilde{\gamma}_{m_k}([t_{m_{k-1}},t_{m_k}])$ is uniformly bounded from below in Euclidean height and that there exists an $M>0$ such that $\abs{a_n}>M$ implies that $n\in \{m_k\}_k$. Since the vectors belonging to $\tilde{A}_R$ are aligned perpendicularly to $\mathbb{C}$, this implies that there is some $D>0$ such that if $R\geq D$, any intersection of $t\mapsto w_\beta(t)$ with $A_R$  at time $\tau$ can only occur when $\tilde{\gamma}_n(\tau)$ is the apex of the geodesic $\tilde{\gamma}_n$ for some $n$. Referring once more to Figure \ref{figure: geodesic arcing between spheres} we see in particular that $a_n$ must satisfy $\abs{a_n}>R-2$. 
    
    Conversely, if $\tilde{\gamma}_n$ has an apex $\tilde{\gamma}_n(\tau)$, it is fairly straightforward to see that $\abs{a_n}\geq R+2$ implies that $w_\beta(\tau)\in A_R$

\begin{lemma}\label{lemma: asympotics for generic model geodesics}
Let $C^*$ be the constant in Proposition \ref{proposition: asymptotic formula for excursion times}. In the setting above, we have that
    \[\abs{\lim_{T\to\infty}\frac{\#\{t\in[0,T]: w_\beta(t)\in A_{R} \}} {T}-\lim_{N\to\infty} \frac{ \sum_{n=0}^{N-1}(\mathbf{1}_{\abs{a_1(z)}\geq R}\circ G_d^n)(\beta) }{C^*N}  } = O\left(\frac{1}{R^3}\right)\]
    for almost all $\beta$.
\end{lemma}
\begin{proof}
We first claim the left hand side exists for almost all $\beta$.
       The map $G_d$ is ergodic with respect to $\mu_d$. By Birkhoff's ergodic theorem, almost all $\beta\in K_d$ satisfy
\begin{equation}\label{equation:Birkhoff for Gauss map}
    \begin{split}
        \lim_{N\to\infty} \frac{\#\{n\in\mathbb{N}: n\leq N\text{ and } \abs{a_n(\beta)}> R\}}{N} &= \lim_{N\to\infty} \frac{ \sum_{n=0}^{N-1}(\mathbf{1}_{\abs{a_1(z)}> R}\circ G_d^n)(\beta) }{N} \\
        &= \mu_d\{z\in K_d: \abs{a_1(z)}> R \}.
    \end{split}
\end{equation}
for all $R$ in the countable set $\{\sqrt{n}:n\in\mathbb{N}\}$. Since $ \abs{a_1(z)}^2$ is always integer-valued, each term in the above equality is constant when $R\in \left(\sqrt{n},\sqrt{n+1}\right)$ for all $n$. Hence \eqref{equation:Birkhoff for Gauss map} holds for all real $R>0$.

By Proposition \ref{proposition: asymptotic formula for excursion times}, we obtain that for almost all $\beta \in K_d$, the limit  $\lim_{T\to\infty}\#\{t\in[0,T]: w_\beta(t)\in A_{R} \}/T$ evaluates to
\[\lim_{N\to \infty}\frac{\#\{t\in[0,t_N^*]: w_\beta(t)\in A_{R} \}} {t_N^*}=\lim_{N\to \infty}\frac{\#\{t\in[0,t_N^*]: w_\beta(t)\in A_{R} \}} {C^*N}  \]
for all $R\in \mathbb{Q}_{>0}$. By \eqref{equation: continuity of peaking times}, the left hand side of the above equation evaluates to the same continuous function in $R$ for almost all $\beta \in K_d$. Since the right hand side is monotone in $R$, the above equation continues to hold almost everywhere for all $R>0$.
    By the discussion preceding the lemma, the left hand side of the equation in the statement of the lemma must be bounded from above by 
    \[\lim_{N\to\infty} \frac{\sum_{n=0}^{N-1} (\mathbf{1}_{R-2<a_1(z)< R+2}\circ G_d^n)(\beta) }{C^*N}\]
    when $R$ is sufficiently large.

The set  $\{z\in K_d:R-2<a_1(z)< R+2\}$ lies in the set $\{z\in K_d:1/(R+3)<z< 1/(R-3)\}$. As in Lemma \ref{lemma: chance of large digits}, it is relatively straightforward to prove that $\mu_d\{z\in K_d:1/(R+3)<z< 1/(R-3)\}=O(1/R^3)$. Hence by the Birkhoff ergodic theorem, the lemma follows.

\end{proof}

 A consequence of Lemma \ref{lemma: pushforward is absolutely continuous} is that for almost all $v\in\UT\mathbf{P}_d$, there exists a model tangent vector $w_\beta$ with the same asymptotic behavior, i.e. there exists some $\Delta$ such that $d(w_\beta(t),v(t+\Delta))\to 0$ as $t\to +\infty$\footnote{In fact, the rate of this convergence is exponential}. We then obtain from Lemma \ref{lemma: asympotics for generic model geodesics} that almost all such $v$ satisfy

 \[\abs{\lim_{T\to +\infty}\frac{\#\{t\in[0,T]: v\in A_{R} \}} {T}-\lim_{N\to +\infty} \frac{ \sum_{n=0}^{N-1} (\mathbf{1}_{a_1(z)\geq R}\circ G^n)(\beta) }{C^*N}  } = O\left(\frac{1}{R^3}\right),\]
 where the existence of the first limit was established in Lemma \ref{lemma: average return times to to A_R}.

By Lemma \ref{lemma: chance of large digits} and \eqref{equation:Birkhoff for Gauss map}, the map $G_d$ satisfies
\[ \lim_{N\to +\infty} \frac{ \sum_{n=0}^{N-1} (\mathbf{1}_{a_1(z)\geq R}\circ G^n)(\beta) }{C^*N} = \frac{H}{C^*R^2},\]
where $H$ is the constant in Lemma \ref{lemma: chance of large digits}. Using lemma \ref{lemma: average return times to to A_R} and letting $R\to +\infty$, we  obtain that
\[  4\frac{\pi \lambda(K_d)}{\text{vol}(\mathbf{P}_d)} = \frac{H}{C^*}.\]

In the proof of Theorem \ref{theorem: Gumbel Law for Bianchi Groups} and Theorem \ref{theorem: cusp excursions on Euclidean Bianchi Orbifolds}, we showed that the constant $\kappa_d$ in Theorem is given by
\[ \frac{1}{2}\log\left(\frac{H}{C^*}\right)-\log2.\]
Since $\lambda(K_d) = \sqrt{d}$, we obtain that 
\[\kappa_d = \frac{1}{2}\log\left(\frac{\pi\sqrt{d}}{\text{vol}(\mathbf{P}_d)}\right).\]

When $d=1$ or $3$, the proofs in this section need to be modified to account for the extra generator, see Proposition \ref{proposition: complete set of generators}. The sets $\tilde{A}_R$ from Lemma \ref{lemma: average return times to to A_R} must be modified to 
\[\tilde{A}_{R}=\{ W(\beta,\alpha,0) \in\UT\mathbf{H}^3: \beta\in K_d', \abs{\alpha}> R\},\] where $K_d'$ is defined in Proposition \ref{proposition: Fundamental Domain for Bianchi Groups} for $d=1,3$. Repeating the proof of Lemma \ref{lemma: average return times to to A_R} then shows that \[\lim_{T\to +\infty}\frac{ \#\{t\in[0,T]: v(t)\in A_{R} \} }{T} = 4\frac{\pi \lambda(K_d')}{\text{vol}(\mathbf{P}_d)R^2}+O\left(\frac{1}{R^3}\right).\] Since $\lambda(K_1')=1/2$ and $\lambda(K_3') = 11/3$, we obtain that 
\[\kappa_1 = \frac{1}{2}\log\left(\frac{\pi}{2\text{vol}(\mathbf{P}_d)}\right)\text{, and } \kappa_2 = \frac{1}{2}\log\left(\frac{\pi}{3\text{vol}(\mathbf{P}_d)}\right).\]

\subsection{Diophantine approximation}

One of Sullivan's motivations for proving his logarithm law \cite{Sullivan} was to apply the following theorem.
\begin{proposition}[Sullivan]
    Let $\psi:(0,\infty)\to [0,1)$ be a function such that for some $r>1$, the ratio \[\frac{\sup_{[y,ry]}\psi}{\inf_{[y,ry]}\psi}\] is bounded for all $y\in(0,\infty)$. Let $d$ be a square-free natural number. For almost any complex number $z$ with respect to the Lebesgue measure. There exist infinitely many solutions to the inequality
    \[\abs{z-\frac{p}{q}} \leq \frac{\psi(q)}{q},\]
    where $p,q\in\mathfrak{o}_d$ are such that there exist $r,s\in\mathfrak{o}_d$ with\footnote{This is equivalent to $p$ and $q$ generating coprime ideals.}
    \begin{equation}\label{equation: Bezout Identity}
        rp+sq=1.
    \end{equation}

\end{proposition}
The proof uses a generalised version of Ford balls, first described in \cite{Ford}. Our approach mirrors \cite{Nakada88}, which uses these Ford balls to prove the following.                                            
\begin{proposition}[Nakada]
    Let $k> 0$.  Let $d$ be a square-free natural number. For almost every $z\in \mathbb{C}$, 
    \[ \lim_{T\to\infty} \frac{ \# \left\{\frac{p}{q}: p,q\in\mathfrak{o}_d \text{ and satisfy \eqref{equation: Bezout Identity}, }\abs{z-\frac{p}{q}} \leq \frac{k}{q^2}\text{ and } \abs{q}\leq T \right \}} {\log T}= L_dk^2,\]
where $L_d>0$ does not depend on $k$ and can be explicitly computed.  
\end{proposition}

Fix a non-square integer $d$. For $k>0$, define the set $F_k:=\{x+iy+rj: r>1/2k\}$. We define the set of  \textbf{generalised Ford Balls} $\mathfrak{F}_k$ as the collection of all balls
\[g\cdot F_k ,\]
where $g\in\PSLo$. If $g= \begin{psmallmatrix}
    a&b\\c&d
\end{psmallmatrix}$, then $g\cdot F_k$ is the Euclidean ball tangent to $\mathbb{C}$ at $\frac{a}{c}$ with radius $k/\abs{c}^2$. By \cite{Ford}, these balls are disjoint if $k<\frac{1}{2}$ and only touch tangentially for $k=1/2$. It follows immediately from the definition that $\mathfrak{F}_k$ is invariant under action by $\PSLo$ and the action is transitive. 

\begin{lemma}\label{lemma: Ford ball criterion}
Let $z\in\mathbb{C}$ and $k>0$. We denote the model geodesic in $\mathbb{H}^3$ defined in Subsection \ref{subsection: excursion times} with attracting endpoint $z$ by $\tilde{\gamma}$. The inequality 
\begin{equation}\label{eq:}
    \abs{z-\frac{p}{q}} \leq \frac{k}{q^2}
\end{equation}

holds for $p,q\in\mathfrak{o}_d$ satisfying \eqref{equation: Bezout Identity} if and only if the generalised Ford ball $g\cdot F_k$ tangent at $p/q$ intersects $\tilde{\gamma}$. Equivalently, if $h\in \PSLo$ satisfies $h(p/q)=\infty$, then the above inequality holds if and only if $h\circ \tilde{\gamma}$ intersects $F_k$. 
\end{lemma}
The lemma follows immediately from the fact that $\tilde{\gamma}$ intersects $g\cdot F_k$ if and only if the distance from z to $p/q$ is less than the radius of the ball.

\begin{proof}[Proof of Theorem \ref{theorem: Diophantine Approximation}]
We use the notation of Section \ref{section: cusp excursions}. 
    By Lemma \ref{lemma: Ford ball criterion},
    \[\abs{z-\frac{p}{q}}\leq \frac{1}{2e^{\kappa_d} \abs{q}^{2}\sqrt{T}u}\] if and only if there is some $h\in \PSLo$ with $h(p/q)=\infty$ such that $h\circ \tilde{\gamma}$ intersects $F_k$ with
    \[k = \frac{1}{2e^{\kappa_d} \sqrt{T} u}.\]
    For large $T$, these correspond to high cusp excursions. By the proof of Lemma \ref{lemma: excursion height is given by digits} in Section \ref{section: intersection times and cusp excursion}, there is an increasing sequence $n_m$ of natural numbers and an $M>0$ for which $n_m-n_{m-1}$ is uniformly bounded for $\beta \in K_d\backslash\mathbb{Q}$, $t_{n_{m-1}}<t_{n_m}$, $\abs{a_n}>M$ implies that $n\in \{n_m\}_{m\in\mathbb{N}}$ and $\gamma_{n_m}(t)$ is uniformly bounded from below in Euclidean height when $t\in [t_{n_{m-1}},t_{n_m}]$ for all $m$.
    
    We showed in the proof of Theorem \ref{theorem: cusp excursions on Euclidean Bianchi Orbifolds} that 
    \[\lim_{T\to\infty} (m\circ\pi^{-1}) \left\{ z\in K_d : \sup_{0 \leq t\leq T} e^{h(w_z(t))}\leq e^{\kappa_d} \sqrt{T}u\right\}=e^{-1/u^2},\]
    where $w_z(t)$ is the unit tangent vector to $\gamma(t)$. The proof clearly works if we replace $m\circ\pi^{-1}$ by any probability measure $\nu$ which is equivalent to the Lebesgue measure. The condition $\sup_{0 \leq t\leq T}e^{h(w_z(t))}\leq e^{\kappa_d} \sqrt{T}u$ is equivalent to the geodesic $\gamma_{n_m}$ intersecting $F_k$ with $k=(2e^{\kappa_d}\sqrt{T})^{-1}$ for some $m$ with $t_{n_m}\leq T$. By Lemma \ref{lemma: bounding the intersection time} the intersection time $t_{n_m}=2\log(q_{n_m})+O(1)$. Hence the condition $\sup_{0 \leq t\leq T}e^{h(w_z(t))}\leq e^{\kappa_d} \sqrt{T}u$ is equivalent to the condition
    \[\abs{z-\frac{p_n}{q_n}}\geq \frac{1}{2e^{\kappa_d} \abs{q}^{2}\sqrt{\log Q}u+ E(n)},\]
    where $q\leq Q=e^{T/2}$ and $E$ is a correction term satisfying $E(n)=O(1)$ as $n\to\infty$. The theorem follows.
    
\end{proof}
\section{Intersection Times and Cusp Excursions}\label{section: intersection times and cusp excursion}
In this section, we fill in some details not included in the previous section for either being lengthy but routine calculations, too technical, or too distracting from the main arguments. 

We first give an explicit formula for the intersection of a geodesic with the unit hemisphere. In what follows, let $H(b)\subset \mathbf{H}^3$ be the hemisphere of radius one with centre $b\in\mathbb{C}$.
\begin{lemma}\label{lemma: intersection of geodesic and unit hemisphere}
    Let $\tilde{\eta}$ be a geodesic with endpoints $(\beta,\alpha)$ satisfying $|\beta|<1$ and $\infty>\abs{\alpha} >1$. The intersection point $z+rj$ of $\tilde{\eta}$ and $H(0)$ satisfies 
    \[r= \frac{\sqrt{(1-|\beta|^2)(|\alpha|^2-1)}}{|\alpha|^2-|\beta|^2}|\alpha-\beta|\text{ and } z= \beta +\frac{1-|\beta|^2}{|\alpha|^2-|\beta|^2}(\alpha-\beta).\]
\end{lemma}
\begin{proof}
    The geodesic $\tilde{\eta}$ is the half-circle from $\beta$ to $\alpha$ perpendicular to the complex plane. This can be parametrised by 
    \begin{equation}\label{equation: parametrising a circle}
        \begin{split}
            [0,1]\to \mathbf{H}: t\mapsto z(t)+r(t)j\text{, where } z(t) &= \beta+t(\alpha-\beta)\text{ and } \\
            r(t)&=\sqrt{t(1-t)}|\alpha-\beta| .
        \end{split}
    \end{equation}
    This will intersect $H(0)$ if and only there exists a $t'\in [0,1]$ satisfying $|z(t')|^2+|r(t')|^2=1$. Since $|z(t')|^2 =z(t')\overline{z(t')}$, we obtain using \eqref{equation: parametrising a circle} that this relation can be written as 
    \[ |\beta|^2+{t'}^2|\alpha-\beta|^2 + 2{t'}\Re((\alpha-\beta)\overline{\beta})+{t'}(1-{t'})|\alpha-\beta|^2=1.\]
    The terms involving ${t'}^2$ cancel. Solving for $t'$ and substituting this into \eqref{equation: parametrising a circle} completes the proof. 
\end{proof}
The following lemma is about the Möbius transformations used to generate $\{\tilde{\gamma}_n\}_n$ in Section \ref{section: cusp excursions}.

\begin{lemma}\label{lemma: iterated Möbius transformations}
    Let $\beta = [0,a_1,a_2,\ldots]$. For $n\in\mathbb{N}$, let
    \[P(n,\beta)= (T^{(-1)^{n-1} a_n}S) \ldots (T^{-a_2}S)(T^{a_1}S) \in \PSLo.\]
    Then 
    \[P(n,\beta) = \begin{pmatrix}
        q_n & -p_n\\
        -(-1)^{n-1}q_{n-1} & (-1)^np_{n-1}
    \end{pmatrix}\cdot(\pm I),\]
    where $p_n/q_n = [0;a_1;\cdots;a_n]$ is the $n$-th convergent of $\beta$, and we recall that each element of $\PSL$ has two representatives in $\SL$ which are additive opposites.
\end{lemma}
\begin{proof}
    Note that \[T^{\pm a}S = \begin{pmatrix}
        \mp a & 1\\
        -1 & 0
    \end{pmatrix}.\]
In particular it is clear the lemma holds for $P(1,\beta)$. The general case follows from induction using the recursive formula for the associated quotient pair of sequences \eqref{equation: convergents}. 
\end{proof}

Given two sequences $(a_n)_n$ and $(b_n)_n$ of positive numbers, we write $a_n\asymp b_n$ if $a_n = O(b_n)$ and $b_n = O(a_n)$ as $n\to \infty$. We establish a bounded equivalence between the intersection times and an expression involving associated quotient pair of sequences. 
\begin{proof}[Proof of Lemma \ref{lemma: bounding the intersection time}]
    Let $n\in\mathbb{N}$. By symmetry of $G_d$ under reflection, we have for almost all $\beta$ that 
    \begin{equation}\label{equation: endpoints under iteration}
        (-1)^n\hat{G}_d^n(\beta,\infty) = (P(n,\beta)\cdot\beta, P(n,\beta)\cdot\infty).
    \end{equation}
    By Lemma \ref{lemma: iterated Möbius transformations} we have that 
    \[(-1)^n\hat{G}_d^n(\beta,\infty) = (-1)^n\left(G_d^n(\beta),\frac{-q_n}{q_{n-1}}\right).\]
    By Lemma \ref{lemma: intersection of geodesic and unit hemisphere} we have that the geodesic $\tilde{\gamma}_n$ intersects the hemisphere $H(0)$ at some point $z_n+r_nj$ with 
    \begin{equation}\label{equation: pilot geodesic intersecting with hemisphere}
        \begin{split}
            r_n&=\frac{\abs{ G_d^n(\beta)+\frac{q_n}{q_{n-1}}}}{\abs{\frac{q_n}{q_{n-1}}}^2 - |G_d^n(\beta)|^2}\sqrt{\left(1-|G_d^n(\beta)|^2\right)\left(\abs{\frac{q_n}{q_{n-1}}}^2-1\right)} \text{ and}\\
            z_n &= (-1)^n\left(G_d^n(\beta)-\frac{1-|G_d^n(\beta)|^2}{\abs{\frac{q_n}{q_{n-1}}}^2 - |G_d^n(\beta)|^2} \left( G_d^n(\beta)+\frac{q_n}{q_{n-1}}\right)\right).
        \end{split}    
    \end{equation}

    This is equivalent to the geodesic $\tilde{\gamma}_0 = P(n,\beta)^{-1}\tilde{\gamma}_n$ intersecting the hemisphere $P(n,\beta)^{-1}H(0)$ at some point 
    \[z' +r'j := P(n,\beta)^{-1}\cdot (z_n+r_nj).\]
    Since this is a point on $\tilde{\gamma}_0$, $z' = \beta$. 
    By Lemma \ref{lemma: iterated Möbius transformations} we see that $P(n,\beta)^{-1} = \begin{pmatrix}
        (-1)^np_{n-1} & p_n\\
        (-1)^nq_{n-1} & q_n
    \end{pmatrix}$,
    so by \eqref{equation: Mobius transformation in coordinates} we have that 
    \begin{equation}\label{equation: Euclidean height of return point}
        r' = \frac{1}{r_n^{-1}|(-1)^nq_{n-1}z_n+q_n|^2+r_n|q_{n-1}|^2}.
    \end{equation}
    Some rearranging of terms shows that $|(-1)^nq_{n-1}z_n+q_n|$ evaluates to
    \begin{equation*}
        \begin{split}
            &|q_{n-1}|\abs{\left(G_d^n(\beta)+\frac{q_n}{q_{n-1}}\right) -\frac{1-|G_d^n(\beta)|^2}{\abs{\frac{q_n}{q_{n-1}}}^2 - |G_d^n(\beta)|^2} \left( G_d^n(\beta)+\frac{q_n}{q_{n-1}}\right)}\\
            &=|q_{n-1}|  \frac{\abs{\frac{q_n}{q_{n-1}}}^2-1}{\left(\abs{\frac{q_n}{q_{n-1}}}^2 - |G_d^n(\beta)|^2\right)} \abs{ G_d^n(\beta)+\frac{q_n}{q_{n-1}}}.
        \end{split}
    \end{equation*}
    Since $G_d^n(\beta)\in K_d$ and $\sup_{z\in K_d}\abs{z}<1$, we obtain from the above equation and \eqref{equation: pilot geodesic intersecting with hemisphere} that
    \begin{equation}\label{equation: two asympotes}
        \begin{split}
                |(-1)^nq_{n-1}z_n+q_n| &\asymp \abs{q_{n-1}}\left(\abs{ \frac{q_n}{q_{n-1}} }^2 -1\right)\abs{\frac{q_{n-1}}{q_n}}\text{ and} \\
                r_n &\asymp \abs{\frac{q_{n-1}}{q_n}}\left(\abs{ \frac{q_n}{q_{n-1}} }^2 -1\right)^{1/2}.
        \end{split}
    \end{equation}
    
    Since $r_n<1$, we obtain by \eqref{equation: Euclidean height of return point} that $1/r'\asymp r_n^{-1}|(-1)^nq_{n-1}z_n+q_n|^2$. Recall that we defined the intersection times $t_n$ in Definition \ref{definition: intersection time} as the time $\tilde{\gamma_n}$ intersects $H(0)$. By substituting the bounded equivalencies in \eqref{equation: two asympotes} we see that
    \[e^{t_{n}} = \frac{1}{r'}\asymp \frac{|q_{n-1}|^3}{{\abs{q_n}}}  \left(\abs{\frac{q_n}{q_{n-1}}}^2-1\right)^{3/2} \]
    as $n\to\infty$. The lemma follows from the identity $\frac{|q_{n-1}|^3}{{\abs{q_n}}} =\abs{q_n}^2\left(\abs{\frac{q_{n-1}}{q_n}}^{2}\right)^{3/2}$.
\end{proof}
The following lemma is a fast growth property of the denominator of continued fraction expansions.
\begin{lemma}\label{lemma: subsequence of multiplicative growth}
    For each $d\in\{1,2,3,5,7,11\}$, there exists a number $r_d>1$ and $M_d\in\mathbb{N}$ such that for each $\beta\in K_d\backslash\mathbb{Q}$, there exists an increasing sequence of integers $(n_k)_{k\in\mathbb{N}}$ satisfying
    \[1\leq n_{k+1}-n_k\leq M_d \text{ and } \abs{\frac{ q_{n_k} }{q_{n_k-1}}} \geq r_d  \]
    for all $k\in \mathbb{N}$.
\end{lemma}
\begin{proof}
    We estimate the numerator $q_n$ of the continued fraction expansion of in terms of $|q_n\beta -p_n|$. We first note the identity 
    \begin{equation}\label{equation: multiplying Gauss map}
        \abs{q_n\beta -p_n} = \abs{\beta}\abs{G_d(\beta)}\abs{G_d^2(\beta)}\cdots \abs{G_d^n(\beta)}.
    \end{equation}
    holds. This can be proved by induction. It can also be proved by considering the transformation $P(n+1,\beta)$ defined in Lemma \ref{lemma: iterated Möbius transformations} and evaluating the derivative of $z\mapsto P(n+1,\beta)\cdot z$ at $z=\beta$ directly and with the chain rule. Another fairly straightforward proof by induction shows that
    \begin{equation}\label{equation: quality of approximation}
        \abs{q_n}\abs{\beta q_n -p_n} = \frac{1}{\abs{G_d^{n+1}(\beta)+\frac{q_{n+1}}{q_n}}}.
    \end{equation}
    We refer to Proposition 3.3 (iii) in \cite{Dani14} for proof. 
    If $A_d=\sup_{z\in K_d} |z|$, we see by the triangle inequality and the reverse triangle inequality that
    \begin{equation}\label{eq: denominators grow like birkhoff product}
        \frac{1}{2\abs{\beta q_n -p_n} }\leq\abs{q_{n+1}}\leq \frac{1}{(1-A_d)\abs{\beta q_n -p_n} }.
   \end{equation}
   From the above and \ref{equation: multiplying Gauss map}, we obtain for all $m,n\in\mathbb{N}$ that
   \begin{equation}\label{equation: bounded equivalence to Birkhoff product}
         \abs{ \frac{q_{n+m}}{q_n} }\geq \frac{1-A_d}{2}\frac{1}{\abs{G_d^{n+1}(\beta)}\cdots \abs{G_d^{n+m}(\beta)}}\geq \frac{1-A_d}{2}A_d^{-m} .
   \end{equation}
Let $1<r_d < A_d^{-1}$.  If the lemma were false, there would be arbitrarily large intervals $\{n,n+1,\ldots, n+m\}$ on which some $\beta\in K_d$ satisfies $\abs{q_{n+1+i}/q_{n+i}}\leq r_d$ for $0\leq i <m$ and hence $\abs{q_{n+m}/q_{n}}\leq r_d^{m}$, which contradicts $\eqref{equation: bounded equivalence to Birkhoff product}$ for $m$ large enough.
\end{proof}

\begin{proof}[Proof of Proposition \ref{proposition: asymptotic formula for excursion times}]
    Recall that we defined the excursion times $t_n^*$ in Definition \ref{definition: excursion times} to be equal to $\max_{m\leq n}t_m$. We wish to prove that for almost all $\beta\in K_d$, the limit $\lim_{n\to\infty}t_n^*$ exists and is equal to a constant $C^*$ independent of $\beta$. 
    We first note that by taking the logarithm in \eqref{equation: multiplying Gauss map} and using \eqref{equation: bounded equivalence to Birkhoff product}, we obtain the relation
    \[\lim_{n\to\infty} \frac{\log \abs{q_n}}{n} = \lim_{n\to\infty}\frac{1}{n}\sum_{k=1}^n\log\left(\abs{G_d^k(\beta)}\right),\]
    provided the latter limit exists. However, since the $\log$ function is integrable on $K_d$ with respect to the Lebesgue measure, it is integrable with respect to $\mu_d$. By the Birkhoff ergodic theorem, almost all $\beta \in K_d$ have denominators $(q_n)_n$ satisfying
    \[\lim_{n\to\infty} \frac{\log \abs{q_n}}{n} = \int \log |z| d\mu_d(z) \in (0,\infty).\]
    It thus suffices to prove that 
    \[\lim_{n\to \infty} \frac{t_n^*}{n}= 2\lim_{n\to\infty} \frac{\log \abs{q_n}}{n} .\]
    Lemma \ref{lemma: subsequence of multiplicative growth} shows that there exists a constant $M_d$ and a sequence $(c(n))_{n}\in\mathbb{N}^\mathbb{N}$ such that for all $n\in\mathbb{N}$ we have $0\leq n-c(n)\leq M_d$  and $\abs{q_{c(n)}/q_{c(n)-1}}\geq r_d$ for some constant $r_d>1$.
    It follows from Lemma \ref{lemma: bounding the intersection time} that
    \[t_{c(n)}\asymp 2\log\abs{q_{c(n)}}\asymp t^*_{c(n)},\]
    where the latter bounded equivalence follows from the fact that $t^*_{c(n)}= t_m$ for some $m<c(n)$, with $t_m$ bounded from above by $D+2\log\abs{q_{m}}$, with $D$ the constant appearing in Lemma \ref{lemma: bounding the intersection time}. Consequently, $\lim_{n\to\infty}\abs{ 2q_{c(n)}}/c(n) =\lim_{n\to\infty}t^*_{c(n)}/c(n)$. 
    However, both $\abs{q_n}$ and $t_n^*$ are monotone sequences, so we may replace $c(n)$ with $n$ in this equality and obtain the proposition.
\end{proof}

\begin{proof}[Proof of Lemma \ref{lemma: excursion height is given by digits}]
    This proof is somewhat technical in nature, so we first briefly sketch the main idea. Recall that for $v,w\in\UT\mathbf{P}_d$, we denote by $d(v,w)$ and $h(v)$ the hyperbolic distance between the base point of $v$ and respectively the base point of $w$, and the projection $\mathcal{H}$ of the horocycle $\{z+j:z\in\mathbb{C}\}\subset \mathbb{H}^3$ to  $\mathbb{P}_d$. By \eqref{equation: iterated Gauss} we have that the Euclidean height of the apex of $\tilde{\gamma}_k$ is approximately equal to $|a_k|/2$, where we define the Euclidean height of a point $z+rj\in \mathbb{H}^3$ to be equal to $r$. The hyperbolic distance to the horocycle $\{z+j:z\in\mathbb{C}\}$ is the logarithm of that. If $\abs{a_k}>2$, then $t_k>t_{k-1}$ and $\tilde{\gamma}_k$ reaches this apex in the time interval $[t_{k-1},t_{k}]$. Intuitively speaking, the lemma should follow by taking the maximal apex for all the first $n$ intersections.
    
    The argument above allows us to understand the behaviour of $\gamma$ only for certain intervals $[t_{k-1},t_{k}]$ corresponding to large values of $\abs{a_k}$. The proof relies on controlling the behaviour of $\gamma$ outside these intervals. We shall accomplish this by constructing an increasing sequence of integers $(m_k)_k$ for which the geodesic segments 
    \[\tilde{\gamma}_{m_k}([m_{k-1},m_k])\] 
    are uniformly bounded from below in Euclidean height, which indeed implies that the distance from $\gamma(t)$ to $\mathcal{H}$ only when $\tilde{\gamma}_{m_k}(t)$ has large Euclidean height for some $k$.

     Fix $r_d>1$ as in Lemma \ref{lemma: subsequence of multiplicative growth} and let $(n_k)_{k}$ be the sequence of all integers satisfying $\abs{q_{n_k}/q_{n_k-1}}\geq r_d$. Define the subsequence $(m_k)_k$ of $(n_k)_k$ by removing the entries $n_k$ for which 
    \[t_{n_k}\leq t_{n_{k-1}},\]
    i.e. $m_k$ is the largest subsequence for which $t_{m_k}$ is increasing\footnote{``largest'' in this context means that any other subsequence satisfying this property is a subsequence of $m_k$.}.
    Note that by Lemma \ref{lemma: bounding the intersection time} we have that 
    \begin{equation}\label{equation: asymptotics for subsequence of intersection times}
        t_{n_k} \asymp  2\log |q_{n_k}|^2
    \end{equation}
    as $k\to +\infty$ for all $\beta\in K_d\backslash \mathbb{Q}$. Since $\abs{q_{n_k}}$ increases exponentially in $k$, the number of consecutive entries that can be removed is bounded, and hence
    \[m_{k+1}-m_k \]
    is a bounded sequence as well. In fact, since the constant $r_d>1$ of Lemma \ref{lemma: subsequence of multiplicative growth} is a lower bound for the exponential growth rate, we can bound this sequence independently of the endpoint $\beta$. Furthermore, there exists an  $M>0$ such that if $\abs{a_n}>M$ for some $n\in\mathbb{N}$, then $\abs{q_n/q_{n-1}}\geq r_d$ and $t_n\geq t^*_{n-1}$. In other words,
    \begin{equation}\label{equation: large digits appear in m_k}
        n\in\{m_k\}_k \text{ if } |a_n|> M. 
    \end{equation}
    
    We first study the behaviour of the geodesic $\gamma$ defined in the introduction to Subsection \ref{subsection: excursion times} on the intervals $[t_{m_{k-1}},t_{m_k}]$. To simplify the notation somewhat, let $\mathfrak{r}=m_{k-1}$ and $\mathfrak{s}=m_k$.
    Take the sequence of lifted geodesics $\{\tilde{\gamma}_n\}_n$ as in Subsection \ref{subsection: excursion times}. We then have
    \begin{equation*}
        \begin{split}
            \tilde{\gamma}_{\mathfrak{s}} = T^{-(-1)^\mathfrak{s} a_\mathfrak{s}}Q_\beta(k)\tilde{\gamma}_{\mathfrak{r}}\text{, where }\\
            Q_\beta(k)=\left(S\prod_{l=\mathfrak{r}+1}^{\mathfrak{s}-1} T^{-(-1)^l a_l}S\right)\tilde{\gamma}_{\mathfrak{r}}.
        \end{split}
    \end{equation*}
    
    
    Let $w_1^*(k)$ be the intersection of $\tilde{\gamma}_{\mathfrak{r}}$ with $H(0)$ at time $t_{\mathfrak{r}}$. On the interval $[t_{\mathfrak{r}},t_{\mathfrak{s}}]=[t_{m_{k-1}},t_{m_k}]$, the geodesic $\tilde{\gamma}_{\mathfrak{s}}$ travels from 
    \[w_1(k):=T^{-(-1)^{\mathfrak{s}} a_{\mathfrak{s}}}Q_\beta(k)w_1^*(k)  \]
    to its intersection $w_2(k)$ with $H(0)$.
    By Lemma \eqref{lemma: intersection of geodesic and unit hemisphere} and the fact that $\abs{q_{\mathfrak{r}}/q_{\mathfrak{r}-1}},\abs{q_{\mathfrak{s}}/q_{\mathfrak{s}-1}}> r_d$, we have that the Euclidean heights of $w_1^*(k)$ and $w_2(k)$ are bounded positively from below for all $k$. The implication \eqref{equation: large digits appear in m_k} means only a finite number of digits can appear in the definition of $Q_\beta(k)$.

    Boundedness of the sequence $m_k-m_{k-1}$ therefore implies that the set \[\{Q_\beta(k): k\in\mathbb{N},\beta\in K_d\backslash\mathbb{Q}\}\] is finite. This shows that the Euclidean height of $w_1(k)$ is bounded from below by a positive number over all values of $k$ and $\beta$. The geodesic segments $\tilde{\gamma}_{m_k}([m_{k-1},m_k])$ are therefore uniformly bounded from below in Euclidean height for all $k\in\mathbb{N}$ and $\beta\notin\mathbb{Q}$. 
    We note that for a point $z'+r'j\in \mathbb{H}^3$ with $r'>1$, the lift of $\mathcal{H}$ which is closest to the point is the horocycle $\{z+j:z\in\mathbb{C}\}$ and the distance between them is $\log r'$. As we previously noted when sketching the proof, we have that if $\abs{a_{m_k}}>2$, the geodesic $\tilde{\gamma}_{m_k}$ reaches its apex at a time $t$ with $t\in [t_{\mathfrak{s}-1},t_{\mathfrak{s}}]\subset[t_{m_{k-1}},t_{m_k}] $. Hence if $z_k+h_kj$ are the apexes of the geodesics $\tilde{\gamma}_{m_k}$, then 

    \[ \abs{\max_{t\in [0,t_{m_k}]} e^{h(w_\beta(t))}-\sup_{l\leq k} h_l }\]
     is uniformly bounded for all $\beta\in K_d\backslash\mathbb{Q}$ and $k\in\mathbb{N}$. The lemma follows from \eqref{equation: large digits appear in m_k} and the fact that $\abs{h_k - \abs{a_{m_k}}/2} \leq 1$.  
\end{proof}

Theorem \ref{theorem: Gumbel Law for Bianchi Groups} is a statement regarding to the asymptotic behaviour of the digits $a_n$ for large $n$, while Theorem \ref{theorem: cusp excursions on Euclidean Bianchi Orbifolds} concerns the asymptotic behaviour of geodesic for large times $t$. The following slightly technical version of Theorem \ref{theorem: Gumbel Law for Bianchi Groups} will allow us to account for errors in converting from $n$ to $t$. 

\begin{proposition}\label{proposition: limit theorem for complex continued fractions but complicated}
    Let $\eta$ be a probability measure on $K_d$ which is absolutely continuous with respect to the Lebesgue measure. Let 
    \[k:\mathbb{N}\times K_d \to \mathbb{N}: (n,\beta)\mapsto k(n,\beta) \]
    be a Borel-measurable function such that $\lim_{n\to\infty}k(n,\beta)/n=1$ for almost all $\beta$. Let $E:\mathbb{N}\times K_d\to\mathbb{R}$ also be a Borel-measurable function which is uniformly bounded outside of a Lebesgue-null set. Then 
    \begin{equation}\label{equation: limit theorem for complex continued fractions but complicated}
            \lim_{N\to\infty} \eta\left\{z\in K_d: \max_{1\leq n\leq k(N,\beta)}|a_n(z)| \leq C_du\sqrt{N}+E(k(N,\beta),\beta) \right\} = e^{-1/u^2} ,
    \end{equation}
 where $C_d$ is the constant in Theorem \ref{theorem: Gumbel Law for Bianchi Groups}.
\end{proposition}
\begin{proof}
     We claim that there is a sequence $k_n\in\mathbb{N}$ such that \[\liminf_{n\to +\infty} \eta\left\{z\in K_d: n-k_n\leq k(n,\beta)\leq n+ k_n \right\}=1\] and $k_n/n\to0 $ as $n\to +\infty$. Indeed, define the function ${k^*}(n,\beta):= \min\left\{\frac{k(n,\beta)}{n},2\right\}$. By Chebyshev's inequality, we have that
     \begin{equation}\label{equation: Chebyshev}
            \eta\left\{z\in K_d: \abs{ {k^*}(n,\beta) -\overline{k}_n^*}\geq a\sigma_n \right\} \leq \frac{1}{a^2}
     \end{equation}
  
    for all $a>0$, where ${\overline{k}_n^*} =\int_{K_d} {k^*}(n,\beta)d\eta(\beta)$ and $\sigma_n^2 =\int_{K_d} {k^*}(n,\beta)^2d\eta(\beta) -{\overline{k}_n^*}^2 $. By the dominated convergence theorem, $\lim_{n\to+\infty}{\overline{k}_n^*}  =1$ and $\lim_{n\to+\infty} \sigma_n=0$. There therefore exists a sequence $a_n$ such that $\lim_{n\to+\infty} a_n =+\infty$ and $\lim_{n\to+\infty}a_n\sigma_n=0$. For $n$ large enough such that $a_n\sigma_n \leq 1 $ and $\abs{{k_n}^*-1}\leq 1$, we have that the condition $\abs{ {k^*}(n,\beta) -\overline{k}_n^*}\geq a\sigma_n$ is equivalent to the condition
    \[\abs{ k(n,\beta) -n\overline{k}_n^*}\geq a\sigma_n n.\]
    Since ${\overline{k}_n^*} $ approaches unity for large $n$, we note that by \eqref{equation: Chebyshev} the claim is satisfied for $k_n=\lceil a_n\sigma_n n +n\abs{k_n^*-1}\rceil$. 

    Let $\epsilon>0$ and let $A>0$ such that $\abs{E}<A$ outside of a null set. By our construction of $k_n$, the
    left hand side of \eqref{equation: limit theorem for complex continued fractions but complicated} is bounded from above by 
     \begin{equation*}
                \limsup_{N\to\infty} \eta\left\{z\in K_d: \max_{1\leq n\leq N-k_N}|a_n(z)| \leq C_du\sqrt{N}+A \right\}.
    \end{equation*}
For sufficiently large $N$, we have that 
\[C_du\sqrt{N}+A \leq C_d(u+\epsilon)\sqrt{\left(N-k_N\right)}\]
and we obtain by Theorem \ref{theorem: Gumbel Law for Bianchi Groups} that the above that left hand side of \eqref{equation: limit theorem for complex continued fractions but complicated} is bounded from above by $e^{-1/(u+\epsilon)^2}$. The lower bound $e^{-1/(u-\epsilon)^2}$ can be proven analogously. We obtain the proposition by letting $\epsilon\to 0$. 
\end{proof}
\begin{proof}[Proof of equality between \eqref{equation: cuspexcursionmanipulations1} and\eqref{equation: cuspexcursionmanipulations2}]

  Let $(t_n^*)_n$ be as in Definition \ref{definition: excursion times}. Proposition \ref{proposition: asymptotic formula for excursion times} states that the time elapsed after $N$ excursions is approximately equal to $NC^*$. 
    If the limits in \eqref{equation: cuspexcursionmanipulations1} exist, then any of them are equal to
\[\lim_{N\to +\infty} (m\circ\pi^{-1})\left\{ \beta\in K_d : \sup_{0 \leq t\leq NC^*} e^{d(w_\beta,w_\beta(t))}\leq e^{\kappa_d} \sqrt{NC^*}u\right\}.\]
Conversely, if the above limit exists and is continuous in $y$, then the limits in \eqref{equation: cuspexcursionmanipulations1} exist and are equal to the above quantity by a similar argument to the last part of the proof of Proposition \ref{proposition: limit theorem for complex continued fractions but complicated}. 
    Define $k(N,\beta)$ to be the largest value of $n'$ such that $t_{n'}^*\leq NC^*$, where $(t^*_n)_n$ is the sequence of excursion times for the geodesic $t\mapsto w_{\beta}(t)$. The quantity $k(N,\beta)$ can be therefore interpreted as the actual number of excursions completed in the expected time it takes to complete $N$ excursions. It follows from the definition that $\lim_{N\to\infty} k(N,\beta)/N = 1$ for almost all $\beta$. Lemma \ref{lemma: excursion height is given by digits} shows that the maximum value of $e^{d(w_\beta,w_\beta(t))}$ attained in these $k(N,\beta)$ excursions is approximately equal to $\frac{1}{2}\max_{1\leq n\leq k(N,\beta)}|a_n(\beta)|$, so we obtain that the expression above is equal to 
    \[\lim_{N\to +\infty} (m\circ\pi^{-1})\left\{ \beta\in K_d : \max_{1\leq n\leq k(N,\beta)}|a_n(\beta)|\leq 2e^{\kappa_d} \sqrt{C^* N}u+ E(k(N,\beta),\beta)\right\},\]
    where $E(N,\beta)$ is the error term in approximating $ \sup_{0 \leq t\leq NC^*} e^{d(w_\beta,w_\beta(t))}$ with $\frac{1}{2}\max_{1\leq n\leq t^*_N}|a_n(\beta)|$.
    By Proposition \ref{proposition: limit theorem for complex continued fractions but complicated} and a rescaling of $y$, we obtain the claim.
\end{proof}
\printbibliography
\end{document}